\documentclass[12pt]{article}
\usepackage{amsmath,amssymb,enumerate}

\usepackage{bef_alex}

\usepackage[latin1]{inputenc}
  \usepackage{color,eucal}
\usepackage{graphicx}
\usepackage{psfrag}
\numberwithin{equation}{section}
\numberwithin{figure}{section}

\renewcommand{\baselinestretch}{1.05}

\newcommand{\foraa}{\text{for a.a.\ }}
\newcommand{\aein}{\text{ a.e.\ in }}
\newcommand{\weakto}{\rightharpoonup}
\newcommand{\down}{\downarrow}

\newcommand{\BV}{\rmB\rmV}
\newcommand{\cont}{\mathrm{co}}
\newcommand{\discr}{\mathrm{J}}
\newcommand{\mespace}{\calX}
\newcommand{\extQ}{{\calQ_T}}
\newcommand{\gph}{\mathbf{q}}

\newcommand{\AC}{\mathrm{AC}}
\newcommand{\VAR}{\mathop{\mathrm{Var}}}
\newcommand{\LIP}{\calG}%{\mathop{\mathrm{Lip}}}
\newcommand{\sppt}{\mathop{\mathrm{sppt}}}
\newcommand{\ITEM}[1]{\ensuremath{\text{\upshape #1}}}
\newcommand{\whT}{S}
\newcommand{\Slope}[3]{|\pl_q #1|\left(#2,#3\right)}
\newcommand{\slope}[2]{|\pl_q #1|(#2)}
\newcommand{\Slopename}[1]{|\pl_q #1|}
\newcommand{\Indicator}[2]{I_{#1}(#2)}
\newcommand{\distS}{S}
\newcommand{\dissS}{\calS}
\newcommand{\DS}[4]{\dissS_{#1}(#2;#3,#4)}
\newcommand{\dS}[4]{\distS_{#1}(#2;#3,#4)}
\newcommand{\bigset}[2]{\Big\{\:#1\:\Big|\:#2\:\Big\}}
\newcommand{\malpha}{m}

\definecolor{ddcyan}{rgb}{0,0.2,1.0}

\newcommand{\OLD}[1]{{\color{red}\sffamily\tiny #1}}
\newcommand{\COMMENT}[1]{{\footnotesize\sffamily\color{red} #1}}
\newcommand{\GCOMMENT}[1]{{\footnotesize\sffamily\color{red} #1}}
\renewcommand{\COMMENT}[1]{}
\renewcommand{\GCOMMENT}[1]{}
\renewcommand{\OLD}[1]{}

\def\RRSS{\color{ddcyan}}
\def\RREE{\color{black}}

\newcommand{\RCOMMENT}[1]{{\footnotesize\sffamily\color{red} #1}}

\renewcommand{\RRSS}{\color{black}}
\renewcommand{\RCOMMENT}[1]{}

\begin{document}

\title{Modeling solutions with jumps \\
for rate-independent systems \\
on metric spaces } %%
\date{19 June 2008 }

\author{
Alexander Mielke\,%
\footnote{\emph{Weierstra\ss-Institut,
    Mohrenstra\ss{}e 39, 10117 D--Berlin and Institut f\"ur
    Mathematik, Humboldt-Universit\"at zu
    Berlin, Rudower Chaussee 25, D--12489 Berlin (Adlershof), Germany}
e-mail: {\tt mielke\,@\,wias-berlin.de}},
Riccarda Rossi\,%
\footnote{\emph{Dipartimento di Matematica, Universit\`a di
 Brescia, via Valotti 9, I--25133 Brescia, Italy,}
 e-mail: {\tt riccarda.rossi\,@\,ing.unibs.it}},
  and
Giuseppe Savar\'e\,%
\footnote{\emph{Dipartimento di Matematica ``F.\ Casorati'',
Universit\`a di Pavia.
 Via Ferrata, 1 -- 27100 Pavia, Italy}
email: {\tt giuseppe.savare\,@\,unipv.it}}
}
\maketitle

\begin{abstract}
Rate-independent systems allow for solutions with jumps that need
additional modeling. Here we suggest a formulation that arises as
limit of viscous regularization of the solutions in the extended
state space. Hence, our \emph{parametrized metric solutions} of a
rate-independent system are absolutely
continuous mappings from a parameter interval into the extended
state space. Jumps appear as generalized gradient flows during which
the time is constant.  The closely related notion of
\emph{BV solutions} is developed afterwards.
Our approach is based on  the abstract theory of generalized
gradient flows in metric spaces, and comparison with other notions of
solutions is given.
\end{abstract}

\section{Introduction}
\label{s:intro}

This paper is concerned with the analysis of different solution
notions for rate-independent evolutionary systems. The latter arise
in a very broad class of mechanical problems, usually in connection
with hysteretic behavior. With no claim at completeness, we may
mention for instance elastoplasticity, damage, the quasistatic
evolution of fractures, shape memory alloys, delamination and
 ferromagnetism, referring to \cite{Miel05ERIS}
%and to the references therein
for a  survey  \RRSS of \RREE the modeling of rate-independent
phenomena.

Because of their relevance in applications, the analysis of these
systems  has attracted some attention over the last decade, also in
connection with the  issue of their proper formulation. In fact, in
several situations rate-independent problems may be recast in the
form of a doubly nonlinear evolution equation involving two {\em
energy functionals}, namely
\begin{equation}
\label{subdif-form} \pl_{\dot{q}}\calR  (q(t), \dot{q}(t)) +
\pl_q \calE(t,q(t)) \ni 0 \qquad \text{in $Q'$} \qquad \foraa\, t \in
(0,T)\,,
\end{equation}
where $Q$ is a separable Banach space, $\calR :Q \times Q
   \to [0,\infty]$   a \emph{dissipation} functional and
$\calE: [0,T ]\times Q  \to  (-\infty,\infty] $ an \emph{energy
potential},
  $\pl_{\dot{q}}$ and $\pl_q $ denoting  their subdifferential with respect to the second variable.
Rate-independence is rendered through $1$-homogeneity of the
functional $\calR $ with respect to its second variable. Indeed,
assuming that $ \calR (q,\gamma v)=\gamma \calR (q,v)$ for all  $
\gamma \geq 0 $ and $(q, v) \in Q \times Q,$ one has that equation
\eqref{subdif-form} is invariant for time-rescalings. This captures
the main feature of this kind of processes, which are driven by  an
external loading set on a time scale much slower than the time scale
intrinsic to the system, but still fast enough to prevent
equilibrium. Typically, this quasistatic behavior originates in the
limit of systems with a viscous, rate-dependent dissipation.

The formulation of rate-independent problems in terms of the
subdifferential inclusion \eqref{subdif-form} has been thoroughly
analyzed in \cite{MieThe04RIHM}, in the case of a reflexive Banach
space.  Existence of solutions to the Cauchy problem for
\eqref{subdif-form} is proved through approximation by time
discretization and solution of incremental minimization problems.
However, in many applications the energy $\calE(t,\cdot)$ is neither
smooth nor convex, and the state space $Q$ is often neither
reflexive nor the dual of a separable Banach space (see e.g.
\cite{KruZim07?EPNR} for energies having linear growth at infinity).
Furthermore, $Q$ may even lack a linear structure (in these cases we
will denote it by the calligraphic letter $\calQ$), like for
finite-strain elastoplasticity \cite{Miel03EFME,MaiMie08?GERI} or
for quasistatic evolution of fractures \cite{DaFrTo05QCGN}.

In such situations, the differential formulation \eqref{subdif-form} cannot be
used.  In \cite{MieThe99MMRI,MiThLe02VFRI}, the concept of \emph{energetic
  solution} for general \emph{rate-independent energetic systems}
$(\calQ,\calE,\calD)$ has been introduced, by replacing the
infinitesimal metric $\calR$ of the subdifferential formulation
\eqref{subdif-form} by a global \emph{dissipation distance} $\calD:
\calQ \times \calQ \to [0,\infty]$.  This formulation, see Section
\ref{ss:en-sln}, is \emph{derivative-free}, and thus applies to
solutions with jumps and can be used in very general frameworks,
\RRSS like, for example,  in a topological space $\calQ$,  with
$\calE$ and $\calD$  lower semi-continuous only, \RREE see
\cite{MaiMie05EREM,Miel05ERIS,FraMie06ERCR}. Energetic solutions are
very flexible and allow for a quite general existence theory;
however, the global stability condition, asking that $q(t)$ globally
minimizes \RRSS the map \RREE $\wt q \mapsto \calE(t,\wt
q)+\calD(q(t),\wt q)$, implies that solutions jump earlier as
physically expected, since they are forced to leave a locally stable
state, see e.g.\ \cite[Ex.\,6.1]{Miel03EFME} or
\cite[Ex.\,6.3]{KnMiZa07?ILMC}, and Example~\ref{ex:5.1} below.
Moreover, existence of energetic solutions is proved via time
discretization and incremental \emph{global minimization}, but, as
discussed in \cite[Sec.\,6]{Miel03EFME}, \emph{local minimization}
would be more appropriate both from the perspective of modeling and
of numerical algorithms.

%\paragraph{\bf The y approach.}
In response to these issues, in \cite{EfeMie06RILS} a
\emph{vanishing viscosity approach} was proposed to derive new
solution types for rate-independent systems $(\calQ,\calR,\calE)$.
There, \RRSS $\calQ$ is assumed to be a \emph{finite-dimensional
Hilbert space} $Q$ \RREE and $\calE \in \rmC^1 ([0,T] \times Q)$.
The natural viscous approximation of \eqref{subdif-form} is obtained
by adding a quadratic term to the dissipation potential, viz.\
$\calR_\eps (q,v) =  \calR (q,v)+\frac{\eps}2 \| v \|^2$, and leads
to the doubly nonlinear equation
\begin{equation}
\label{viscous-subdif}
\eps \dot{q}(t) + \pl_{\dot{q}}\calR  (q(t), \dot{q}(t)) +
\pl_q \calE(t,q(t)) \ni 0 \qquad \foraa\, t \in (0,T)\,.
\end{equation}
Using dim$\!\;Q <\infty$, the  existence of solutions $q_\eps\in
\rmH^1([0,T];Q)$ is obvious and passing to the limit $\eps \searrow
0$ in \eqref{viscous-subdif} leads to new solutions and to a finer
description of the jumps, which occur later than for energetic
solutions. The key idea (see Section~\ref{s:mech.mot}) is that the
limiting solution at jumps shall follow a path which somehow keeps
track of the viscous approximation.
%In other words, the viscous potential $\calR_{\eps}$
%provides crucial information in order to characterize the jump
%path of the limiting solution.
To exploit this additional information, one has to go over to an
extended state space: reparametrizing the approximating viscous
solutions $q_\eps$ of \eqref{viscous-subdif} by their arclength
${\tau}_\eps$, and introducing the rescalings $\wh{t}_\eps =
\tau_\eps^{-1}$ and $\wh{q}_\eps= q_\eps \circ \wh{t}_\eps$, one
studies the limiting behavior of the sequence $\{( \wh{t}_\eps,
\wh{q}_\eps) \}_\eps $ as $\eps \down 0$.  Hence, in
\cite{EfeMie06RILS} it was proved that (up to a subsequence), $\{(
\wh{t}_\eps, \wh{q}_\eps) \}_\eps $ converges to a pair
$(\wh{t},\wh{q})$, whose evolution encompasses both dry friction
effects and, when the system jumps, the influence of rate-dependent
dissipation. In fact, the jump path may be completely described by a
gradient flow equation, which leads to this interpretation: jumps
are fast (with respect to the slow external time scale) transitions
between two metastable states, during which the system switches to a
viscous regime.  Furthermore, solutions of the limiting
rate-independent problem can be constructed by means of a
time-discretization scheme featuring local, rather than global,
minimization.

This paper provides the first step of the generalization of these
ideas to the much more general metric framework using the concept
\RRSS of \RREE \emph{curves of
  maximal slope},  \RRSS which dates back \RREE to the pioneering
paper \cite{DeMaTo80PEMS}. We also refer to the recent
monograph \cite{AmGiSa05GFMS}, the references therein, and to
\cite{RoMiSa08MACD}.
The general setup starts with a
\[
\text{complete metric space $(\mespace ,d)$}
\]
and introduces the \emph{metric velocity}
\begin{equation}
\label{e:1-mder}
 |q'|:= \lim_{h \searrow 0}\frac{d(q(t),q(t{+}h))}{h}=\lim_{h \searrow
   0}\frac{d(q(t{-}h),q(t))}{h}\,,
\end{equation}
which is defined a.e.\ along an absolutely continuous curve $q:[0,T]\to
\mespace$.

This replaces the norm of the derivative \RRSS $q'$ \RREE in the
smooth setting, and, in the same way, the  norm of the
(G\^{a}teaux)-derivative or the subdifferential of a functional
$\calF: \mespace  \to (-\infty,\infty]$ is replaced by the
\emph{local slope} of $\calF$ in $q\in \mathrm{dom}(\Psi)$, which
\RRSS is \RREE defined by
\begin{equation}
\label{1-local-slope}
 \slope\calF q:= \limsup_{v \to q}
\frac{\left(\calF(q) -\calF(v)\right)^+}{d(q,v)}\,,
\end{equation}
\RCOMMENT{Here I replaced $\Psi$ by $\calF$ because we use $\Psi$
for denoting the dissipation functional, and I thought that this
might be a bit confusing, as we always consider the slope of the
energy functional} where $(\cdot)^+$ denotes the positive part. With
these concepts, the viscous problem \eqref{viscous-subdif} has the
equivalent metric formulation
\begin{equation}
\label{e:mform} \frac{\rmd}{\rmd t}\calE (t,q(t))-
  \pl_t\calE (t,q(t)) \leq -(|{q}'|(t)+ \frac\eps{2}|{q}'|^2(t))
  -\frac1{2\eps}\left(\big(\Slope\calE t{q(t)}-1\big)^+\right)^2,
\end{equation}
for a.a.\ $t \in (0,T )$, see Section~\ref{sez3.1} for further details.
It was proved in \cite[Thm.\,3.5]{RoMiSa08MACD} that, under
suitable assumptions on $\calE$, for every $\eps>0$ the related Cauchy
problem has at least one solution $q_\eps \in \AC([0,T];\mespace )$.

 Following the approach of \cite{EfeMie06RILS}, for every $\eps >0$
 we now consider the
 (arclength) rescalings $(\wh{t}_\eps, \wh{q}_\eps)$ associated with
 $q_\eps$, which in turn fulfill a rescaled version
 of \eqref{e:mform}, cf.\ \eqref{e3.11}.
   \OLD{ which fulfill the rescaled  version of \eqref{e:mform}
 \begin{equation}
 \label{rescale-eps}
 \begin{aligned}
\frac{\rmd}{\rmd s}&\calE (\wh{t}_\eps(s),\wh{q}_\eps(s))
-\pl_t\calE (\wh{t}_\eps(s),\wh{q}_\eps(s))\wh{t}_\eps'(s)
 \\ & \leq
-\left(|\wh{q}_\eps'|(s)+\frac\eps{2\wh{t}_\eps'(s)}|{q}'|^2(s)
\right) -
\frac{\wh{t}_\eps'(s)}{2\eps}\left(\big(\Slope\calE t{q(t)}-1\big)^+\right)^2\,,
\end{aligned}
\end{equation}
almost everywhere on some interval $ (0,\wh{T}_\eps)$.} Under suitable
assumptions, in Theorem~\ref{3.3} we shall show that, up to a subsequence,
$\{(\wh{t}_\eps,\wh{q}_\eps)\}$ converges to a limit curve $(\wh{t}, \wh{q})
\in \AC ([0,\whT]; [0,T] \times \mespace )$ such that
\begin{subequations}
\label{mprif}
\begin{align}
\label{mprifa}
&\begin{aligned}
&  \text{$\wh{t}: [0,\whT] \to [0,T]$ is nondecreasing, } \\
&  \wh{t}'(s) + |\wh{q}'|(s) >0 \quad \foraa\, s \in [0,\whT],
\end{aligned}
\\[0.5em]
\label{mprifb}
&\begin{aligned}
& \left.\ba{@{}ccc} \wh{t}'(s)>0&\Longrightarrow&
 \Slope\calE {\wh{t}(s)}{\wh{q}(s)}\leq 1,\\
|\wh{q}'|(s)>0&\Longrightarrow&
  \Slope\calE {\wh{t}(s)}{\wh{q}(s)}\geq 1,
 \ea \right\} \text{ for a.a.\ }s \in [0,\whT]\,,
\end{aligned}
\\[0.1em]
\intertext{and the \emph{energy identity}}
\label{mprifc}
&\begin{aligned}
 \frac{\rmd}{\rmd s}\calE (\wh t(s),\wh q(s))
  &-\pl_t\calE (\wh t(s),\wh q(s))\, \wh t'(s) =
   \langle\rmD_q\calE (\wh{t}(s),\wh{q}(s)),\wh{q}'(s)\rangle\\
  & =-|\wh{q}'|(s)\, \Slope\calE{\wh{t}(s)}{\wh{q}(s)} \
   \ \foraa  s\in(0,\whT)\,,
\end{aligned}
\end{align}
\end{subequations}
holds.  A pair $(\wh{t}, \wh{q}):[s_0,s_1]\to [0,T]\ti \mespace$
satisfying \eqref{mprif}
(with $[0,S]$ replaced by $[s_0,s_1]$)
is called \emph{parametrized metric solution} of the
rate-independent system $(\mespace ,d,\calE)$.

Indeed, the very focus of this paper is on getting insight into the
properties of parametrized metric solutions and comparing them with
the other solution notions for rate-independent evolutions. That is
why, in order to avoid technicalities and to highlight, rather, the
features of our approach, throughout the next sections we shall work
in a technically simpler setup, in which the state space $\mespace $
is a finite-dimensional manifold, endowed with a (Finsler) distance
$d$ associated with a $1$-homogeneous dissipation functional
$\calR:\rmT\calQ\to [0,\infty)$, and an energy \RRSS $\calE \in
{\rmC}^1( [0,T] \times \calQ )$. \RREE \RCOMMENT{Here I used $[0,T]
\times \calQ $ as we haven't yet defined $\extQ$.}
The fully general
metric framework is postponed to the forthcoming paper
\cite{MiRoSa08?VVLM}, see also Section~\ref{s:outlook}.

The notion of parametrized metric solution $(\wh{t},\wh{q})$
generalizes the outcome of the finite-dimensional vanishing
viscosity analysis of \cite{EfeMie06RILS} and hence allows for the
same mechanical interpretation, see Remark~\ref{rem:mechanical}.
Namely, according to whether either of the derivatives $\wh{t}'$ or
$|\wh{q}'|$  is null or strictly positive, one distinguishes in
\eqref{mprifb} three regimes: sticking, rate-independent evolution,
and switching to a viscous regime (in correspondence to jumps of the
system from one metastable state to another). In this metric setup
as well, we show that the behavior of the system along a jump path
is described by a generalized gradient flow.  This can be seen more
clearly when considering the non-parametrized solution $q$
corresponding to the pair $(\wh{t},\wh{q})$. The latter functions
are called \emph{BV solutions} of $(\calQ ,d,\calE)$ and are
\emph{pointwise limits of the un-rescaled} vanishing viscosity
approximations $q_\eps$, see Definition~\ref{def:BVsln}  and
Section~\ref{ss:BB.sol} for an analysis of their properties. In
particular, we shall show how to pass, by  means of a suitable
transformation, from a  (truly jumping) $\BV $ solution $q$ to a
(``virtually'' jumping) parametrized solution $(\wh{t},\wh{q})$, and
conversely.

In Section~\ref{s:other.sol}, we compare the notion of BV solutions
with other solutions concepts, namely with the {energetic} solutions
of \cite{MieThe99MMRI,MiThLe02VFRI}, and with the {approximable} and
{local} solutions of \cite{KnMiZa07?ILMC,ToaZan06?AVAQ,Cagn08?VVAF}
(suitably rephrased in the metric setting, see
Definitions~\ref{d:4.1}, \ref{d4.2}, and~\ref{d4.5}). In
Section~\ref{ss:Vis.sol} we review the notion of
\emph{$\Phi$-minimal solutions} of a rate-independent evolutionary
system, proposed in \cite{Visi01NAE} using a \emph{global
variational principle} in terms of a suitably defined partial order
relation between trajectories. First, we conclude that the notion of
local solution is the most general concept, including energetic and
BV solutions, whereas BV solutions \RRSS  encompass \RREE
approximable and $\Phi$-minimal solutions. Moreover, our notion of
BV solutions has ``more structure'', which makes it robust with
respect to data perturbations (cf.
Remark~\ref{rem:point-forward-stability}), whereas neither
approximable nor $\Phi$-minimal solutions are upper-semicontinuous
with respect to data perturbations.

Further insight into the comparison between the various solution
notions is provided by the examples presented in
Section~\ref{s:examples}, which are one or two-dimensional, such
that the set of all solutions can be discussed easily. The
one-dimensional case in fact relates to crack growth (under the
assumption of a prescribed crack path), which was treated in
\cite{ToaZan06?AVAQ,Cagn08?VVAF,NegOrt07?QSCP,KnMiZa07?ILMC}. The
solution concepts developed there are also based on the vanishing
viscosity method. In the latter case, the solution type  \RRSS in
fact \RREE  coincides with our notion of BV solution.  We
postpone the more difficult PDE applications to
\cite{MiRoSa08?VVLM}, where we are going to combine the notions of
this paper with the methods of \cite{RoMiSa08MACD} to develop the
present ideas in the infinite-dimensional or fully metric setting.
 Related ideas using the vanishing viscosity method for PDEs
are found in  \cite{DDMM07?VVAQ}, for a model for elastoplasticity
problems with softening, and  in \cite{MieZel08?VVLP}, for general
parabolic PDEs with rate-independent dissipation terms.

\section{Setup and mechanical motivation}
\label{s:mech.mot}

We consider a manifold $\calQ $ that contains the states of our
system. The energy $\calE $ of the system depends on the time
$t\in[0,T]$ and the state $q\in\calQ $.
Throughout the paper, we shall assume that $\calE \in\rmC^1(\extQ )$, where
$\extQ=[0,T]\ti\calQ$ denotes the extended state space.
The evolution of the system is governed by a balance between the potential
restoring force $-\rmD_q\calE (t,q)$ and a frictional force $f$. The latter is
given by a continuous dissipation potential $\calR :\ \rmT\calQ \to[0,\infty)$, in the
form $f\in\pl_{\dot{q}}\calR (q,\dot{q})$. We generally assume that $\calR
(q,\cdot):\ \rmT_q\calQ \to[0,\infty)$ is convex and $\pl_{\dot{q}}\calR
(q,\dot{q})\subset\rmT^*_q\calQ $ is the set-valued subdifferential. Hence,
the system is governed by the differential inclusion
\begin{equation}\label{e2.1} 0
  \in\pl_{\dot{q}}\calR (q(t),\eps\dot{q}(t))+\rmD_q\calE (t,q(t))
   \subset\rmT_q^*\calQ \,,\qquad t \in (0,T)\,,
\end{equation}
in which we have introduced a small parameter $\eps >0$ to
indicate that we are on a very slow time scale.

Further, we suppose  that $\calR = \calR_1 +
\calR_2$, where $\calR_1: \
\rmT\calQ \to[0,\infty)$ and $\calR_2 : \
\rmT\calQ \to[0,\infty)$ are such that for every $q \in
\calQ $
\[
\begin{aligned}
\calR_1(q,\cdot) \text{ is convex and homogeneous of degree } 1,
\\
\calR_2(q,\cdot) \text{ is convex and homogeneous of degree } 2.
\end{aligned}
\]
Note that $\calR_j(q,\gamma v)=\gamma^j\calR (q,v)$ implies $\pl\calR_j(q,
\gamma v)=\gamma^{j-1}\pl\calR_j(q,v)$ for all $\gamma \geq 0$ and $(q,v) \in
\rmT\calQ $. Hence, \eqref{e2.1} takes the form
\begin{equation}\label{e2.2}
0\in\pl_{\dot{q}}\calR_1(q(t),\dot{q}(t))+\eps\pl_{\dot{q}}\calR_2(q(t),\dot{q}(t))+\rmD_q\calE (t,q(t))\,,\qquad
t \in (0,T). \end{equation}  We call $\calR_1$ the potential of
rate-independent friction and $\calR_2$ the potential of
viscous friction.
%
% Perfect: Alex agrees (15.4.08) !!
% \COMMENT{Here I preferred not to specify the
% assumptions on $\calR_1$, I think it would be better to
% specify them latter, right before we introduce the  Finsler distance
% $d$ associated with $\calR_1$...}
%

\begin{remark}
\upshape A prototype of the mechanical situation we aim to model
arises in connection with a system of $k$ particles ($k \geq 1$)
moving in $\R^d$, hence with state-space $ \calQ  = \left\{q=
(q_1, \ldots, q_k) \, : \ q_i \in \R^d \right\}= \R^{kd}\,. $ We
impose rate-independent friction $\calR_1$ and viscous friction $\calR_2$ via
\[
\calR_1 (q, \dot{q}) = \sum_{j=1}^k \mu(q_j) |\dot{q}_j|
\ \text{ and } \  \calR_2 (q, \dot{q}) = \sum_{j=1}^k \frac{\nu(q_j)}2
|\dot{q}_j|^2
\quad \text{for } (q, \dot{q}) \in \R^{kd} \times \R^{kd}\,,
\]
where $|\dot{q}_j|$ is the Euclidean norm of the $j$-th particle
velocity and $\mu, \, \nu: \R^d \to [0,\infty)$ are given
continuous
functions. For $k=1$ the potentials $\calR_j$ are
related by $\calR_2 (q, \dot{q}) = \frac{\nu(q)}{2\mu^2 (q)}\,
\calR_1^2 (q, \dot{q})$,
while for $k=2$ their interplay is more complex.
\end{remark}

 Our aim is to understand the
limiting behavior of the solutions to \eqref{e2.2} for $\eps \to 0$.
In fact, we expect that for $\eps\to 0$  the rate-independent
friction dominates,  but the solution $q^\eps:\ [0,T]\to\calQ $ may
develop sharp transition layers, with $\dot{q}$ of order $1/\eps$.
In the limit we obtain a jump, but in order to characterize the jump
path the viscous potential is crucial.

The key idea  is to study the trajectories $\mathcal{T}_\eps
=\set{(t,q^\eps(t))}{t\in[0,T]}$ in the extended state space
$\extQ$. The point is that the limit of trajectories
$\mathcal{T}_\eps$ may no longer be the graph of a function. To
study the limits via differential inclusions,  we may reparametrize
the trajectories $\mathcal{T}_\eps$ in the form
\[
\mathcal{T}_\eps
=\set{({\wh{t}}_\eps(s),{\wh{q}}_\eps(s))}{s\in[0,S_\eps]},
\]
where ${\wh{t}}_\eps$ is supposed to be nondecreasing and
absolutely continuous.

For passing to the limit it is now helpful to select a family of
parametrizations via
$\malpha_\eps\in\rmL^1_{\text{loc}}((0,\infty))$ converging to
$\malpha$ in $L^1_{\text{loc}}((0,\infty))$, with
$\malpha(s),\malpha_\eps(s)>0$ for a.a. $s \in (0,\infty)$, and to
assume
\begin{equation}\label{e2.3}
{\wh{t}}_{\eps}'(s)+\sqrt{2\calR_2({\wh{q}}_\eps(s),{\wh{q}}_{\eps}'(s))}=
\malpha_\eps(s)
\qquad \foraa\, s \in (0,S_\eps)\,. \end{equation}
Note that this can always
be achieved. In particular, when
\begin{equation}
\label{e:particular}
 \calQ  = \R^d\,, \quad
 \calR_2 (q, \dot{q})=
\frac{1}2 |\dot{q}|^2 \ \ \forall\, (q, \dot{q}) \in \R^d \times
\R^d\,  \quad \text{and $\malpha=\malpha_\eps \equiv 1$\,,}
\end{equation}
 relation \eqref{e2.3}
leads to the arclength parametrization of $\mathcal{T}_\eps$, which
was considered in \cite{EfeMie06RILS}. The total length is
%\begin{equation}
%  \label{eq:4}
 $ S_\eps:=T+\int_0^T \sqrt{2\calR_2(q_\eps(t),q_\eps'(t))}\, dt$.
%\end{equation}
Since $S_\eps \to \whT$ up to a subsequence (thanks to standard
energy estimates), it is not restrictive to assume that $S_\eps$ is
independent of $\eps$ by the simple linear rescaling
%\begin{equation}  \label{eq:9}
$  \tilde\malpha_\eps(s)=\malpha_\eps(s S_\eps/S) $.
%\end{equation}

By the chain rule and the $j$-homogeneity we have
\[
\pl_{\dot{q}}\calR_j(q_\eps(t),\dot{q}_\eps(t))|_{t={\wh{t}}^\eps(s)}=
\left( \frac{1}{{\wh{t}}'_\eps(s)}\right)^{j-1}
\pl_{\dot{q}}\calR_j(\wh{q}_\eps(s),{\wh{q}}'_\eps (s)) \qquad
\foraa\, s \in (0,S_\eps)\,.
\]
Now, using \eqref{e2.3} and defining
\[
\wt{\calR }(q,v)=g(\calR_2(q,v))\quad\text{with }g(r)=
\left\{ \ba{cl}
\log \left( \frac{1}{1-\sqrt{2r}}\right) - \sqrt{2r}&\text{for }
r\in[0,\frac{1}2), \\
\infty & \text{otherwise,} \ea\right.
\]
easy computations (cf. \cite[Thm.\,3.1]{EfeMie06RILS} in the
particular case of \eqref{e:particular}) show that \eqref{e2.2} is
equivalent to \begin{equation}\label{e2.4}
0\in\pl_{\dot{q}}\calR_1(\wh{q}_\eps ,{\wh{q}}_\eps
')+\eps\pl_{\dot{q}}\wt{\calR }\Big(\wh{q}_\eps ,\frac{{\wh{q}}_\eps
'}{\RRSS\malpha_\eps \RREE}\Big)+\rmD_q\calE (\wh{t}_\eps
,\wh{q}_\eps) \quad \aein \, (0,S_\eps)\,. \end{equation}
In this formulation  we may pass to the limit, and  we expect to
obtain the limit problem
\begin{equation}\label{e2.5}
\left. \ba{l}
0\in\pl_{\dot{q}}\wh{\calR }(\wh{q},\frac{{\wh{q}}'}{\malpha})+
\rmD_q\calE (\wh{t},\wh{q}),\\
{\wh{t}}'+\sqrt{2\calR_2(\wh{q},{\wh{q}}')}=\malpha, \ea
\right\} \quad \aein \, (0,\whT)\,,
\end{equation}
\RRSS where \RREE
\[
\wh{\calR }(q,v)=
\left\{\ba{cl}
\calR_1(q,v)&\text{for }\calR_2(q,v)\leq \frac12\,,\\
\infty &\text{for } \calR_2(q,v)>\frac12\,. \ea \right.
\]
Formulation \eqref{e2.5} is in fact a generalization of the one in
\cite{EfeMie06RILS}, where rigorous convergence proofs  of
problem \eqref{e2.4} to \eqref{e2.5} are derived (in the case
$\calQ $ is finite-dimensional). Although $\wh{\calR }$ is
no longer $1$-homogeneous, the limit problem is still
rate-independent: upon adjusting the free function $\malpha$, one
sees immediately that system \eqref{e2.5} is invariant under time
reparametrizations.

In the present work we concentrate on the case $\calR_2(q,v)=
\frac12\calR_1(q,v)^2$, since it is this case which can be generalized to
abstract metric spaces and hence to the infinite-dimensional setting, see
\cite{MiRoSa08?VVLM}.  By introducing the dual norm of a  co-vector
$w\in\rmT^*_q\calQ$
\begin{equation}
  \label{dual_nrom}
  \calR_{1,*}(q,w):=\sup\big\{\langle w,v\rangle:v\in \rmT_q\calQ,
  \ \calR_1(q,v)\le 1\big\},
\end{equation}
the operators $\pl\calR_1(q,\cdot)$ and
$\pl\calR_2(q,\cdot)$ can be characterized by
\begin{subequations}
  \begin{align}
    \label{duality_op:1}
    w\in \pl\calR_1(q,v),\ v\neq0\quad&\Longleftrightarrow\quad
    \calR_{1,*}(q,w)=1,\ \langle w,v\rangle=\calR_1(q,v)>0\,,\\
     \label{duality_op:1bis}
    w\in \pl\calR_1(q,0)\quad&\Longleftrightarrow\quad
    \calR_{1,*}(q,w)\le 1\,,\\
    \label{duality_op:2}
    w\in \pl\calR_2(q,v)\quad&\Longleftrightarrow\quad
         \calR_{1,*}(q,w)=\calR_1(q,v)=\langle w,v\rangle\,,
  \end{align}
 \end{subequations}
and they satisfy $\pl\calR_2(q,v)=\calR_1(q,v) \,\pl\calR_1(q,v)$ and
\begin{equation}
  \label{duality_op:3}
  \calR_{1,*}(q,w)\,\calR_1(q,v)=\langle w,v\rangle
  \quad\Longleftrightarrow\quad
   w \in \lambda\pl\calR_1(q,v)
   \quad\text{for some }\lambda\ge0.
\end{equation}
\begin{proposition}\label{p:2.1}
In the case $\calR_2(q,v)=\frac12\calR_1(q,v)^2$, a pair
$(\wh{t},\wh{q}) \in \AC ([0,\whT]; [0,T] \times
\calQ )$ fulfils \eqref{e2.5} (for some $\malpha \in L^1
(0,\whT)$ with $\malpha(s)>0$ \aein $(0,\whT)$) if and only if there exists a function $\lambda :
(0,\whT) \to (0,\infty)$ such that \begin{equation}\label{e2.6} \left. \ba{l}
0\in\lambda\pl\calR_1(\wh{q},{\wh{q}}')+\rmD_q\calE (\wh{t},\wh{q}),\\
{\wh{t}}'\geq 0,\ \lambda\geq 1,\ (\lambda{-}1){\wh{t}}'\equiv 0,
\\
{\wh{t}}' + \calR_1(\wh{q},{\wh{q}}') >0 \ea \right\} \qquad
\aein \, (0,\whT)\,. \end{equation}
\end{proposition}
\begin{proof}
First, we note that, using the $1$-homogeneity of $\calR_1$,
the second of \eqref{e2.5} may be rewritten as
\begin{equation}
\label{e2.6-1} \frac{{\wh{t}}'}{\malpha}+\calR_1
\left(\wh{q},\frac{{\wh{q}}'}{\malpha} \right)=1 \qquad \aein
(0,\whT)\,.
\end{equation}
Now, it is not difficult to see that
\begin{equation}
\label{e2.6-2}
\pl_{\dot{q}}\wh{\calR }(\wh{q},\frac{{\wh{q}}'}{\malpha}) =
\left\{\ba{cl}
\pl_{\dot{q}}\calR_1(\wh{q},\frac{{\wh{q}}'}{\malpha}) &\quad
\text{if $\calR_1(\wh{q},\frac{{\wh{q}}'}{\malpha}) \in [0,1)
\quad ( \Leftrightarrow \, {\wh{t}}' >0)\,,$}
\\
{}[1, \infty) \, \cdot\,
\pl_{\dot{q}}\calR_1(\wh{q},\frac{{\wh{q}}'}{\malpha}) &\quad
\text{if $\calR_1(\wh{q},\frac{{\wh{q}}'}{\malpha})=1  \qquad
(\Leftrightarrow \, {\wh{t}}' =0)$\,,}
\ea\right.
\end{equation}
where the equivalences in parentheses follow from \eqref{e2.6-1}.
 Combining \eqref{e2.6-2} with the first of \eqref{e2.5} and using that
$\pl \calR_1$ is $0$-homogeneous,
  we
 deduce \eqref{e2.6}.

 Conversely, starting from \eqref{e2.6}, we put $\malpha(s):= {\wh{t}}'(s) +
 \calR_1(\wh{q}(s),{\wh{q}}'(s)) $ for a.a.\ $s \in (0,\whT)$ and note
 that, by the third of \eqref{e2.6}, $\malpha >0$ a.e.\ in $(0,\whT)$
 and $\malpha\in L^1(0,\whT)$, since $\calR_1$ is continuous.
 Using \eqref{e2.6-1} and arguing as in the above lines, one sees that, if the
 pair $({\wh{t}}', \lambda)$ satisfies the second of \eqref{e2.6}, then $
 \lambda\pl_{\dot{q}}\calR_1(\wh{q},{\wh{q}}') = \pl_{\dot{q}}\wh{\calR
 }(\wh{q},\frac{{\wh{q}}'}{\malpha})\,, $ which allows us to deduce the
 differential inclusion in \eqref{e2.5} from the one in \eqref{e2.6}.
\end{proof}

\section{Analysis with metric space techniques}
\label{sez3}

\subsection{Problem reformulation in a metric setting}
\label{sez3.1}

First of all, we complement the setup of Section~\ref{s:mech.mot} by
specifying our assumptions on the rate-independent system
$(\calQ,\calR_1,\calE)$, where $\calQ$ is the ambient space,
$\calR_1$ the dissipation functional, and $\calE$ the energy
functional. The more general setup will be studied in
\cite{MiRoSa08?VVLM}. Namely, we require that
\begin{equation}
\label{assq} \tag{\thesection.Q}
\calQ \text{ is a finite-dimensional and smooth
 manifold,}
\end{equation}
and the energy functional satisfies
\begin{equation} \label{assene} \tag{\thesection.E}
\calE \in \rmC^1 (\extQ)\,.
\end{equation}
The dissipation functional $\calR_1: \rmT\calQ \to[0,\infty)$ is  a
\emph{complete Finsler structure} on $\calQ $ (see e.g.\
\cite[Ch.\,I.1]{BCS00IRFG}), namely
\begin{equation}
\label{asserre} \tag{\thesection.R}
\calR_1 \text{ is continuous on } \rmT\calQ \quad \text{ and }\quad
\forall\, q \in \calQ:\ \calR_1(q,\cdot) \text{ is a norm on }
\rmT_q\calQ \,,
\end{equation}
called \emph{Minkowski norm} in the Finsler setting.
Then, $\calR_1$ induces the (Finsler) distance $d:\calQ\ti \calQ \to
[0,\infty)$:
\begin{displaymath}
d(q_0,q_1):=\min\Big\{\int^1_0\calR_1(\wt{q}(s),{\wt{q}}'(s))\dd s:  \wt{q}\in
\calA(q_0,q_1)\Big\}\,,
\end{displaymath}
where for all $q_0,q_1\in \calQ$ we set
\begin{equation}
\calA(q_0,q_1)=\set{y\in\AC([0,1],\calQ)}{y(0)=q_0,\
  y(1)=q_1}.\label{eq:ACqq}
\end{equation}
 Hence,  $(\calQ,d)$ is a metric space, which we assume to be \emph{complete.}
Like in the previous section, we let $\calR_2 \equiv
\frac12\calR_1^2$. \COMMENT{I collected the three assumptions
\eqref{assq}, \eqref{asserre},
  \eqref{assene} and added the section number, since it will be difficult to
  spot them otherwise.}
%
%Choosing  a suitable local coordinate system  in $\calQ $, we
%have (see \cite[Chap.VI.2]{BCS00IRFG})
%\begin{equation}
%\label{crucial-exp}
%d(q,\wt{q})=\calR_1(q,\wt{q}{-}q)+O(|\wt{q}{-}q|^2)\,.
%\end{equation}
% Henceforth, for the ambient space we shall use  the
%notation $\calQ$, instead of $\calQ $, whenever we aim to
%highlight the metric  structure induced by $d$ on $\calQ $.
%\COMMENT{If you do not agree with this, it can be easily fixed
%because I have used a macro for $\calQ$} Accordingly, in this Finsler
%setting we now see how the usual derivative notions relate to the
%metric derivatives we mentioned in the Introduction.
%
For a curve $q \in \AC ([0,T]; \calQ)$, the
Finsler length of its velocity $q'(t)$ is given by
\begin{equation}
\label{newform1}
 |q'|(t):=\calR_1(q(t),q'(t)), \ \text{
 well-defined for a.a.\ $t \in (0,T)$,}
\end{equation}
and satisfies
\begin{equation}
\label{e:mder} d(q(s),q(t)) \leq \int_s^t |q'|(r)\, \rmd r \qquad
\text{for all } 0 \leq s < t \leq T\,.
\end{equation}
Using $\calR_1$, we define the associated \emph{local slope}
$\Slopename\calE :\extQ\to [0,\infty]$  via
\begin{equation}
\label{newform2} \Slope\calE tq:=\sup_{v\in
\rmT_q\calQ \setminus\{0\}}\frac{\langle\rmD_q\calE
(t,q),v\rangle}{\calR_1(q,v)}
=\calR_{1,*}(q,\rmD_q\calE
(t,q))\,,
\end{equation}
which is the conjugate norm with respect to the Minkowski norm
$\calR_1(q,\cdot)$  of the differential of the energy in the
cotangent space $\rmT^*_q\calQ$.

Using the smoothness of $\calE$ we have that for every curve $(t,q) \in \AC
([s_0,s_1]; \extQ)$ the map $s \mapsto \calE (t(s),q(s))$ is absolutely
continuous and the chain rule for $\calE$ gives
\begin{equation}
\label{classical-ch-rule} \frac{\rmd}{\rmd
s}\calE (t(s),q(s))=\pl_t\calE (t(s),q(s))\,
t'(s)+\langle\rmD_q\calE (t(s),q(s)),q'(s)\rangle
\end{equation}
$\foraa\, s \in (s_0,s_1).$ On the other hand, formulae
\eqref{newform1} and \eqref{newform2}  yield
\begin{equation}
\label{useful-later} \langle\rmD_q\calE (t,q),q')\rangle\geq
-|q'|\ \Slope\calE tq\,.
\end{equation}
Therefore, every $(t,q) \in \AC ([s_0,s_1]; \extQ )$ fulfills the
chain rule inequality
\begin{equation}
\label{e3.6}
\frac{\rmd}{\rmd s}\calE (t(s),q(s))
  -\pl_t\calE (t(s),q(s))t'(s) \,\geq \,
  -\Slope\calE{t(s)}{q(s)}\,|q'|(s) \ \aein (s_0,s_1).
\end{equation}

%%%%%%%%%%%%%%%%

\paragraph{The metric formulation of doubly nonlinear equations.}

We now see how notions \eqref{newform1} \RRSS and \RREE
\eqref{newform2} so far introduced come into play in the
reformulation of a class of doubly nonlinear evolution equations in
the metric setting $(\calQ,d)$.

Let $\psi:[0,\infty)\to[0,\infty]$ be a lower semicontinuous, nondecreasing,
and convex function and $\psi^*:[0,\infty)\to[0,\infty]$ its conjugate
function (Legendre--Fenchel transform), namely
\[
\psi^*(\xi)=\sup\set{\nu\xi-\psi(\nu)}{\nu\geq 0}.
\]
Following \cite{AmGiSa05GFMS} (see also \cite{RoMiSa08MACD}), a
function $q\in\AC([0,T];\calQ)$ is called a solution of the
$\psi$-gradient system associated \RRSS with \RREE $(\calQ,d,\calE)$
if
\begin{equation}\label{e3.9}
 \frac{\rmd}{\rmd t}\calE (t,q(t))\leq
  \pl_t\calE (t,q(t))-\psi(|q'|(t))-
 \psi^*\big(\Slope\calE t{q(t)}\big)
\quad \aein (0,T)\,.
\end{equation}
It has been proved in \cite[Prop.\,8.2]{RoMiSa08MACD} that $q$ fulfills
\eqref{e3.9} if and only if it solves the doubly nonlinear equation (also
called quasi-variational evolutionary inequality)
\begin{equation}
\label{dne} 0
\in\pl_{\dot{q}}\Psi(q(t),\dot{q}(t))+\rmD_q\calE (t,q(t))
\aein  (0,T)\,, \quad
\text{where }\Psi(q,\dot{q}):= \psi(\calR_1 (q,\dot{q}))\,.
\end{equation}
Under assumptions \eqref{assq}, \eqref{asserre}, and \eqref{assene},
the existence of absolutely continuous solutions to the Cauchy
problem for \eqref{e3.9} follows from
\cite[Thm.\,3.5]{RoMiSa08MACD}. We stress that the simple, but
central duality inequality $\psi(\nu) + \psi^*(\xi) \geq \nu \xi$
for all $\nu,\, \xi \in [0,\infty)$, together with the chain rule
inequality \eqref{e3.6}, enforces equality in \eqref{e3.9}
(ultimately in \eqref{e3.6} as well).

In the rate-independent setting, the natural choice is
\[
\psi_0(\nu)\equiv\nu  \quad
\text{ giving } \quad \psi_0^*(\xi)=\Indicator{[0,1]}\xi\,,
\]
where $ \Indicator{[0,1]}\cdot$ denotes the indicator function of $[0,1]$,
i.e.\ $ \Indicator{[0,1]}\xi= 0 $ if $\xi \in [0,1]$, and $
\Indicator{[0,1]}\xi= \infty $ otherwise. However, simple one-dimensional (not
strictly convex) examples show that we cannot expect existence of absolutely
continuous solutions in this case, cf.\ Example \ref{ex:5.1}. Hence, we
proceed as in Section \ref{s:mech.mot} and consider limits of viscous
regularizations after suitable reparametrizations.

Before doing so, note that \eqref{e3.9} is equivalent to the
parametrized version on some interval~$(s_0,s_1)$, given by
\begin{equation}\label{e3.10}
\ba{l}
\frac{\rmd}{\rmd s}\calE (\wh{t}(s),\wh{q}(s))-\pl_t\calE (\wh{t}(s),\wh{q}(s))\wh{t}'(s)\\
\leq -\psi\Big(\frac{1}{\wh{t}'(s)}|\wh{q}'|(s)\Big)\wh{t}'(s)
-\psi^*\Big(\Slope\calE {\wh{t}(s)}{\wh{q}(s)}\Big)\wh{t}'(s)
\quad \foraa\, s \in (s_0,s_1) \,,
\ea
\end{equation}
where $\wh{q}(s)=q(\wh{t}(s))$ and $\wh{t}'(s)>0$ a.e.\ in $(s_0,s_1)$. In
the rate-independent case,  the right-hand side does not depend on
$\wh{t}'(s)$, because $\psi_0(\nu)=\nu$ implies
$\psi_0(\alpha\nu)=\alpha\psi_0(\nu)$ and
$\psi_0^*(\xi)=\alpha\psi_0^*(\xi)$ for all $\alpha>0$.

\subsection{Rate-independent limit of viscous metric flows}
\label{ss3.2}

We now consider the case of small viscosity added to
the rate-independent dissipation, namely
\begin{equation}
\label{e:small-viscosity}
 \psi_\eps(\nu)=\nu+\frac{\eps}{2}\nu^2 \qquad \forall\, \nu \in
 [0,\infty)\,.
 \end{equation}
We obtain $\psi^*_\eps(\xi)=\frac{1}{2\eps}((\xi{-}1)^+)^2$. Thus,
\eqref{e3.10} takes the form
\begin{equation}\label{e3.11}
\frac{\rmd}{\rmd s}\calE (\wh{t}(s),\wh{q}(s))
  -\pl_t\calE (\wh{t}(s),\wh{q}(s))\,\wh{t}'(s) \leq
-M_\eps\big(\wh{t}'(s),|\wh{q}'|(s),
\Slope\calE{\wh{t}(s)}{\wh{q}(s)}\big)
\end{equation}
for a.a. $s \in (s_0,s_1)$, with
\begin{equation}\label{e3.11+1}
M_\eps\big(\alpha,\nu,\xi\big):=
\alpha\psi_\eps\Big(\frac{\nu}{\alpha}\Big)+\alpha\psi^*_\eps(\xi)=\nu+\frac{\eps}{2\alpha}
{\nu}^2+\frac{\alpha}{2\eps}((\xi{-}1)^+)^2
 \end{equation}
for all $(\alpha,\nu,\xi) \in (0,\infty) \times [0,\infty)^2$.
\COMMENT{DELETE here and move into proof of Theorem \ref{3.3}:
As we have already mentioned,  by the chain rule \eqref{e3.6}, eqn.\
\eqref{e3.11} in fact holds as an equality, which we write in the
integral form
\begin{equation}
\label{int-form}
%\begin{aligned} &
\ts
\calE (\wh{t}(\sigma_2),\wh{q}(\sigma_2)) +\int_{\sigma_1}^{\sigma_2}
M_\eps\big(\wh{t}'(s),|\wh{q}'|(s),
\Slope\calE{\wh{t}(s)}{\wh{q}(s)}\big) \dd s
%  \\ &\quad
= \calE (\wh{t}(\sigma_1),\wh{q}(\sigma_1))   +
\int_{\sigma_1}^{\sigma_2}
\pl_t\calE (\wh{t}(s),\wh{q}(s))\,\wh{t}'(s) \dd s,
%\end{aligned}
\end{equation}
where $\,s_0 \leq \sigma_1 \leq \sigma_2 \leq s_1\,$.}
 Clearly, for fixed $\alpha>0$ the limit as $\eps\searrow 0$ of $M_\eps (\alpha,\nu,\xi)$ gives
$\psi_0(\nu)+\psi^*_0(\xi)$. However, our purpose is to blow-up the
time parametrization whenever jumps occur. Indeed,  the
finite-dimensional case (see \cite{EfeMie06RILS}) suggests that
jumps in the rate-independent evolution will occur at  fixed
rescaled time (i.e., when  $\wh{t}'=0$).
 Hence, we also have to consider the case $\alpha\to 0$ as $\eps\to 0$. For this, note that when $\xi>1$
 $M_\eps(\cdot,\nu,\xi)$ assumes its minimum on $[0,\infty)$ for
 $\alpha^\eps_*=\eps\nu / (\xi-1)^+$,
 corresponding to the
 value $M_\eps\big(\alpha^\eps_*,\nu,\xi\big)=\nu+\nu(\xi{-}1)^+$.
 In any case we have
 \begin{equation}\label{eq3.Meps}
   M_\eps\big(\alpha,\nu,\xi\big)\geq
   M_\mathrm{inf}(\nu,\xi):=\nu+\nu(\xi{-}1)^+ \
   \text{ for all } (\alpha,\nu,\xi) \in (0,\infty) \times [0,\infty)^2\,.
\end{equation}

Thus, we define $M_0: [0,\infty)^3 \to [0,\infty]$ via
\begin{equation}
\label{ex_1_M_0}
 M_0\big(\alpha,\nu,\xi\big):= \left\{\ba{cl}
   M_\mathrm{inf}(\nu,\xi)=
   \nu+\nu(\xi{-}1)^+&\text{for }\alpha=0,\\
   M_\mathrm{sup}(\nu,\xi)=
   \nu+\Indicator{[0,1]}\xi & \text{for }\alpha>0, \ea \right.
\end{equation}
and obtain the following result.

\begin{lemma} \label{l3.M-Gamma} Define $M_\eps:[0,\infty)^3\to [0,\infty]$ via
  \eqref{e3.11+1},  $M_\eps(0,0,\xi)=0$ for all $\xi$, and
  $M_\eps(0,\nu,\xi)=\infty$ for all $\xi$ and $\nu>0$.
  Then, we
  have the following results:
\medskip

\noindent (A)  $M_\eps$ $\Gamma$-converges to $M_0$, viz.
\begin{subequations}
\label{glimit}
\begin{align}\label{gliminf}
\begin{split}
&\text{$\Gamma$-liminf estimate: }\\
&(\alpha_\eps,\nu_\eps,\xi_\eps)\to
 (\alpha,\nu,\xi) \ \Longrightarrow M_0(\alpha,\nu,\xi)\leq
 \liminf_{\eps \searrow 0} M_\eps(\alpha_\eps,\nu_\eps,\xi_\eps)\, ,
\end{split} \\[0.3em]
\label{glimsup}
\begin{split}
&\text{$\Gamma$-limsup estimate: }\\
&\forall\,  (\alpha,\nu  , \xi)\
   \exists \, ((\alpha_\eps,\nu_\eps, \xi_\eps)  )_{\eps>0}:\
  \left\{\ba{l}(\alpha_\eps,\nu_\eps  , \xi_\eps )\to
    (\alpha,\nu, \xi  ) \text{ and }\\
    M_0(\alpha,\nu,\xi)\geq \limsup_{\eps \searrow 0}
    M_\eps(\alpha_\eps,\nu_\eps  , \xi_\eps  ) \, .
  \ea\right.
\end{split}
\end{align}
\end{subequations}
(B) If $(\alpha_\eps ,\nu_\eps) \weakto (\wh\alpha, \wh\nu) $ in
$\rmL^1((s_0,s_1))$ and\/ $\liminf_{\eps\to0} \xi_\eps(s)\ge \wh\xi(s)$
\aein $(s_0,s_1)$,
then
\[
\int_{s_0}^{s_1} M_0(\wh\alpha(s),\wh\nu(s),\wh\xi(s)) \dd s \leq \liminf_{\eps \to 0}
\int_{s_0}^{s_1} M_\eps(\alpha_\eps(s),\nu_\eps(s),\xi_\eps(s)) \dd s\,.
\]
\end{lemma}
\begin{proof} Estimate \eqref{gliminf} is trivial for $\alpha>0$, as we have
  pointwise convergence then. If $\alpha=0$, we employ \eqref{eq3.Meps} and
  use that $M_\mathrm{inf}$ is continuous.

To obtain \eqref{glimsup} in the case $\alpha>0$ we simply take
$(\alpha_\eps,\nu_\eps,\xi_\eps) =(\alpha,\nu,\xi)$ and the result
follows from pointwise convergence. If $\alpha=0$, we let
$(\alpha_\eps,\nu_\eps,\xi_\eps)= (\alpha^\eps_*,\nu,\xi)$ and the
desired result follows. Thus, (A) is proved.

To show the estimate in Part (B),
 let us introduce the function $\bar M:[0,\infty)^4\to [0,\infty]$
\begin{displaymath}
  \bar M(\alpha,\nu;\xi,\eps):=M_\eps(\alpha,\nu,\xi);
\end{displaymath}
by the previous point (A) and the fact that $\bar M$ is lower
semicontinuous when $\eps>0$, it is immediate to check that $\bar M$
is lower semicontinuous in $[0,\infty)^4$. Moreover, $\bar
M(\cdot,\cdot;\xi,\eps)$ is convex in $[0,\infty)^2$ for all
$\xi,\eps$: this property can be directly checked starting from the
definition of $\bar M$ or by observing that \RRSS $
M_\eps(\cdot,\cdot,\xi)$ \RREE is convex when $\eps>0$ thanks to
\eqref{e3.11+1} and the convexity of the  \RRSS map  $(\nu,\alpha)
 \mapsto \nu^2/\alpha$. \RREE

Assuming initially that $\xi_\eps\to\wh \xi$ in $L^1(s_1,s_2)$ and
considering an arbitrary infinitesimal subsequence $\eps_n\to 0$, we
can then apply Ioffe's Theorem (see  \cite{Ioff77LSIF}) to the
sequence of maps $s\mapsto
(\alpha_{\eps_n}(s),\nu_{\eps_n}(s),\xi_{\eps_n}(s),\eps_n)$,
obtaining
\begin{displaymath}
  \int_{s_0}^{s_1} \bar M(\wh\alpha(s),\wh\nu(s),\wh\xi(s),0) \dd s
  \leq \liminf_{n\to \infty}
  \int_{s_0}^{s_1} \bar
  M(\alpha_{\eps_n}(s),\nu_{\eps_n}(s),\xi_{\eps_n}(s),\eps_n)
  \dd s\,.
\end{displaymath}
\RREE In the general case, we consider an arbitrary $\kappa>0$ and
we replace $\xi_{\eps_n}$ with the sequence
$\xi_{\kappa,n}(s):=\min(\xi_{\eps_n}(s),\wh\xi(s),\kappa)$,
converging to $\wh\xi_k(s):= \min(\wh\xi(s),\kappa)$ in
$L^1(s_1,s_2)$. Since $\bar M$ is nondecreasing with respect to
$\xi$, we argue as above and obtain
\begin{align*}
  \int_{s_0}^{s_1} \bar M(\wh\alpha(s),\wh\nu(s),\wh\xi_\kappa(s),0) \dd s
  &\leq \liminf_{n\to \infty}
  \int_{s_0}^{s_1} \bar M(\alpha_{\eps_n}(s),\nu_{\eps_n}(s),\xi_{\kappa,n}(s),\eps_n)
  \dd s
  \\&\leq \liminf_{n\to \infty}
  \int_{s_0}^{s_1} \bar M(\alpha_{\eps_n}(s),\nu_{\eps_n}(s),\xi_{\eps_n}(s),\eps_n)
  \dd s\,.
\end{align*}
Passing to the limit as $\kappa\to\infty$ and applying Fatou's Lemma
we obtain the desired inequality.
% denote the integral on the right-hand side
% by $I_\eps$ and set $a=\liminf_{\eps \to 0}I_\eps$. For $a=\infty$
% nothing is to be shown. For $a<\infty$ we may assume $I_\eps \leq a+1$ for
% $0<\eps<\eps_0$. Thus, we have the estimate $\int_{s_0}^{s_1}
% \alpha_\eps((\xi_\eps{-}1)^+)2 \dd s \leq 2\eps (a{+}1)$. Since $\alpha_\eps$
% converges weakly and $\xi_\eps$ uniformly, we obtain in the limit
% $\int_{s_0}^{s_1} \wh\alpha((\wh\xi{-}1)^+)2 \dd s =0$, which implies
% $\wh\alpha(s)(\wh\xi(s){-})^+=0$ a.e.\ in $(s_0,s_1)$.
% Thus, introducing the function $M^*:[0,\infty)^3\to [0,\infty];\:
% (\alpha,\nu,\xi) \mapsto \nu+\nu(\xi{-}1)^+=\nu\max\{1,\xi\}$ we have
% $M^*(\alpha,\nu,\xi)=M_0(\alpha,\nu,\xi)$ whenever $\alpha(\xi{-}1)^+=0$,
% which is the case along the limit $(\wh\alpha,\wh\nu,\wh\xi)$ a.e.. Moreover,
% \eqref{eq3.Meps} gives  $M_\eps\geq M^*$ and we obtain
% \[
% \ba{l}
% \int_{s_0}^{s_1} M_0(\wh\alpha(s),\wh\nu(s),\wh\xi(s)) \dd s =
% \int_{s_0}^{s_1} M^*(\wh\alpha(s),\wh\nu(s),\wh\xi(s)) \dd s \\
% = \int_{s_0}^{s_1} \wh\nu(s) \max\{1,\wh\xi(s)\} \dd s \stackrel{(1)}{=}
% \lim\limits_{\eps \to 0} \int_{s_0}^{s_1} \nu_\eps(s) \max\{1,\xi_\eps(s)\} \dd s \\
% =\lim\limits_{\eps\to 0}\int_{s_0}^{s_1} M^*(\alpha_\eps(s),\nu_\eps(s),\xi_\eps(s)) \dd s
% \leq \liminf\limits_{\eps\to 0}\int_{s_0}^{s_1}
% M_\eps(\alpha_\eps(s),\nu_\eps(s),
% \xi_\eps(s)) \dd s,
% \ea
% \]
% where in $(1)$ we used the weak convergence of $\nu_\eps$ and the uniform
% convergence of $\xi_\eps$. Thus, (B) is established.
\end{proof}

\RREE In fact, the specific form of $M_0$ is not needed in the
sequel. Hence, we will consider general functions
$M:[0,\infty)^3\to[0,\infty]$ with the  following properties (which
are obviously satisfied by $M_0$):
\begin{subequations}
\label{m-prop}
\begin{align}
\label{Mprop1}
& M: [0,\infty)^3 \to [0,\infty]  \text{ is l.s.c.,}
\\
\label{Mprop2}
& M(\gamma \alpha, \gamma \nu,\xi) = \gamma  M(\alpha, \nu,\xi)
\quad \text{ for all } \alpha, \nu,\xi,\gamma  \in
[0,\infty)\,,
\\
\label{Mprop3}
& M(\alpha, \nu,\xi) \geq \nu \xi \quad \text{ for all } \alpha, \nu,\xi \in
[0,\infty)\,,
\\
\label{Mprop4}
& M (\alpha, \nu,\xi\big) = \nu \xi \quad \Longleftrightarrow\quad
(\alpha,\nu,\xi)\in \Xi \,,
\end{align}
\end{subequations}
where the set
$\Xi:=\Xi^{\text{stick}}\cup \Xi^{\text{slip}} \cup
\Xi^{\text{jump}}$ consists of the disjoint flat pieces (see Figure \ref{fig.Xi})
\begin{equation}
\label{def_Xi} \ba{l}
\Xi^{\text{stick}}:=\set{(\alpha,0,\xi)\in[0,\infty)^3}{\alpha\geq 0, \xi\in[0,1)},\\
\Xi^{\text{slip}}:=\set{(\alpha,\nu,1)\in[0,\infty)^3}{\alpha>0,
  \nu\geq 0},  \ \ \text{and}\\
\Xi^{\text{jump}}:=\set{(0,\nu,\xi)\in[0,\infty)^3}{\nu\geq 0,
\xi\geq 1}. \ea
\end{equation}
For instance, the
function
\begin{equation}
  \label{eq:6}
  \wt{M}(\alpha,\nu,\xi)=\nu+(\xi{-}1)^+(\alpha{+}\nu)=
  \max(\xi,1)\nu+ (\xi{-}1)^+\alpha
\end{equation}
fits in this framework. It is not difficult to check that, if $M$ is
nondecresasing with respect to $\xi$, then $M\le M_0$.
\begin{figure}
\psfrag{111}{$1$} \psfrag{XXX}{$\xi$} \psfrag{AAA}{$\alpha$}
\psfrag{NNN}{$\nu$} \psfrag{Sstick}{$\Xi^{\text{stick}}$}
\psfrag{Sjump}{$\Xi^{\text{jump}}$} \psfrag{Sslip}{$\Xi^{\text{slip}}$}
\centerline{\includegraphics[width=0.5\textwidth]{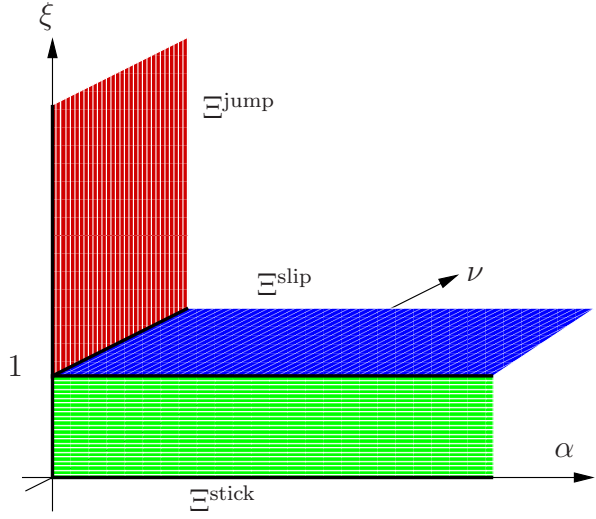}}
\caption{The set $\Xi=\Xi^{\text{stick}}\cup \Xi^{\text{slip}} \cup
\Xi^{\text{jump}}$}
\label{fig.Xi}
\end{figure}

\COMMENT{Relying on \eqref{glimit}, in Theorem~\ref{3.3} we
are going to show that, if $\{(\wh{t}_\eps,\wh{q}_\eps)\} \subset
\AC([s_0,s_1];\extQ)$ is a sequence
fulfilling \eqref{integral-form} for every $\eps>0$, up to a
subsequence $\{(\wh{t}_\eps,\wh{q}_\eps)\}$ converges to a curve
$(\wh{t},\wh{q})$ satisfying
\begin{align*}
\int_{\sigma_1}^{\sigma_2} & M_0\big(\wh{t}'(s),|\wh{q}'|(s),
\Slope\calE{\wh{t}(s)}{\wh{q}(s)}\big)\, \rmd s  +
\calE (\wh{t}(\sigma_2),\wh{q}(\sigma_2))
\\ & =
\calE (\wh{t}(\sigma_1),\wh{q}(\sigma_1))   +
\int_{\sigma_1}^{\sigma_2}
\pl_t\calE (\wh{t}(s),\wh{q}(s))\,\wh{t}'(s)\, \rmd s \quad
\forall\,s_0 \leq \sigma_1 \leq \sigma_2 \leq s_1\,.
\end{align*}
This integral identity gives rise to}

\paragraph{The notion of parametrized metric solutions.}

The following notion of \emph{parametrized metric solution of the
  rate-independent system $(\calQ,d,\calE)$} is proposed in a general form,
replacing the function $M_0$ obtained in the vanishing viscosity
limit with a generic function $M$ satisfying \eqref{m-prop}.
The proposed notion is fitted to the metric \RRSS framework \RREE
and does not need a differentiable structure. However, it strongly
relies on the fact that the small viscous friction $\eps |q'|(t)^2$
is given in terms of the same metric velocity as the
rate-independent friction, see also the assumption $\calR_2=
\frac12\calR_1^2$. We refer to \RRSS
\cite{EfeMie06RILS,MieZel08?VVLP} \RREE for settings avoiding this
assumption.  \RCOMMENT{I have removed the reference
\cite{MiRoSa08?VVLM} because whenever we refer to the metric paper
throughout this paper, it seems that there we will just work with
one (quasi)-distance.. since we still don't know what we are going
to do in the metric case, maybe at this point we might refer to the
paper in the Banach space case, the one with two norms.. or keep the
reference as it is..}

% \GCOMMENT{1) Do we need surjectivity of $\wh t$ in the definition of
% parametrized solutions?\\
% 2) The non-degeneracy condition \eqref{e.DefPaMeSln_b}
% is natural but non strictly necessary, since any possibly degenerate
% solution can be reparametrized to obtain a non degenerate one;
% It is not clear to me what we loose if we do not assume
% \eqref{e.DefPaMeSln_b}.
% 3) Our notion of solution also depends on the viscosity approximation:
% other approximation (e.g. with a different norm) could lead to different solutions, which still could be called parametrized metric solution
% of $(\calQ,d,\calE)$.
% Should we stress this point? }

\begin{definition}[Parametrized metric solution]
\label{def3.2}
Let %the rate-independent system
$(\calQ,d,\calE)$ satisfy
\eqref{assq}, \eqref{asserre}, and \eqref{assene} and
let $M$ fulfill \eqref{m-prop}. An absolutely continuous curve
$(\wh{t},\wh{q}):(s_0,s_1)\to \extQ$ is called a
\emph{parametrized metric solution} of  $(\calQ,d,\calE )$, if
\begin{subequations}\label{e.DefPaMeSln}
\begin{align}
\label{e.DefPaMeSln_a}
&\wh{t}:(s_0,s_1)\to[0,T]\quad\text{is nondecreasing},\\
\label{e.DefPaMeSln_b}
&\wh{t}'(s)+|\wh{q}'|(s)>0\quad\foraa s\in(s_0,s_1),
\\
\label{e.DefPaMeSln_c}
&\begin{aligned} &\frac{\rmd}{\rmd s}\calE (\wh{t}(s),\wh{q}(s))-
               \pl_t\calE (\wh{t}(s),\wh{q}(s))\,\wh{t}'(s)\\
      & \leq - M \big(\wh{t}'(s),|\wh{q}'|(s),
  \Slope\calE {\wh{t}(s)}{\wh{q}(s)}\big)\ \text{for a.a.\
  }s\in(s_0,s_1)\,.
  \end{aligned}
\end{align}
\end{subequations}
 If $(\wh t,\wh q)$ satisfies only \eqref{e.DefPaMeSln_a} and
\eqref{e.DefPaMeSln_c}, it is called a \emph{degenerate parametrized
metric
  solution}. If $\wh{t}:(s_0,s_1)\to[0,T]$ is also surjective, i.e.\
$\wh{t}(s_0)=0$ and $\wh{t}(s_1)=T$, then $(\wh t,\wh q)$ is called a
\emph{surjective parametrized metric solution}.
\end{definition}
 This solution concept has the concatenation property as well
as the restriction property. The former means that if $(\wh t,\wh
q):(s_0,s_1)\to \extQ$ and $(\wt t,\wt q):(s_1,s_2)\to \extQ$ are
parametrized metric solutions with $(\wh t(s_1^-),\wh q(s_1^-))
=(t_1,q_1)=(\wt t(s_1^+),\wt q(s_1^+))$, then \RRSS their \RREE
concatenation $(t,q):(s_0,s_2)\to \extQ$ is a solution as well.
 We point out that, thanks to \eqref{Mprop2}, the notion of
parametrized metric solution is rate-independent, i.e., invariant
under time reparametrizations by  absolutely continuous
functions  \RRSS with strictly positive derivative a.e. \RREE (by
nondecreasing absolutely continuous functions in the case of
degenerate solutions).  Moreover, the notion is independent of
the particular choice of $M$, as long as $M$ satisfies
\eqref{m-prop}.

\begin{remark}[Nondegeneracy and arclength reparametrization]
  \label{rem:nonde} \OLD{Alex has shortened a little!}
  \upshape
  Any  degenerate
  parametrized metric solution \OLD{satisfying \eqref{e.DefPaMeSln_a}
  and \eqref{e.DefPaMeSln_c}} admits a nondegenerate reparametrization
  $(\tilde t,\tilde q):[0,\tilde S]\to \extQ$,
  thus satisfying also \eqref{e.DefPaMeSln_b}. This means that
  $\wh t(s)=\tilde t(\sigma(s)), \wh q(s)=\wh q(\sigma(s))$ for some
  absolutely continuous, nondecreasing and surjetive map
  $\sigma:[s_0,s_1]\to [0,\tilde S]$. In particular, we can
  choose $\sigma$ so that $\tilde t'+|\tilde q'|=1$
  a.e.\ in $(0,\tilde S)$ by defining
  \begin{equation*}
    \label{eq:10}
    \sigma(s):=\int_{s_0}^s \big(\wh t'(s)+|\wh q'|(s)\big)\dd s=
    \wh t(s)-\wh t(s_0)+ \int_{s_0}^s |\wh  q'|(s)\dd s,
    \quad
    \tilde S:=\sigma(s_1)
  \end{equation*}
 (cf.\ also  Lemma \ref{le:chain}). In fact,
  for every interval $[r_0,r_1]\subset [s_0,s_1]$ we then have
  \begin{displaymath}
    \RRSS \sigma(r_0)=\sigma(r_1) \RREE \quad
    \Leftrightarrow\quad
    \wh t(r_0)=\wh t(r)=\wh t(r_1),\quad
    \wh q(r_0)=\wh q(r)=\wh q(r_1)\text{\ for all }r\in
    [r_0,r_1].
  \end{displaymath}
  We can then define $(\tilde t(\sigma),\tilde q(\sigma)) :=(\wh t(s),
   \wh q(s))$, whenever $\sigma=\sigma(s)$. For $\sigma_0=\sigma(r_0)<
  \sigma_1=\sigma(r_1)$ we obtain
  \begin{displaymath}
    \tilde t(\sigma_1)-\tilde t(\sigma_0)+d(\tilde q(\sigma_1),
    \tilde q(\sigma_0))\le \int_{r_0}^{r_1}
    \big(\wh t'(s)+|\wh q'|(s)\big)\dd
    s=
    \sigma_1-\sigma_0\,,
  \end{displaymath}
  giving $\tilde t'+|\tilde q'|\le 1$ a.e.\ in $[0,\tilde S]$.
  The nondegeneracy condition  holds with
  $\tilde t'+|\tilde q'|=1$ a.e.\ in $\tilde S$,  which follows via
  the change of variable formula:
  \begin{align*}
    \tilde S&\ge
    \int_0^{\tilde S} \big(\tilde t'+|\tilde q'|\big)\dd\sigma=
    \RRSS  \int_{s_0}^{s_1} \RREE \big(\tilde t'(\sigma(s))+
    |\tilde q'|(\sigma(s))\big)\sigma'(s)\dd s\\&=
    \int_{s_0}^{s_1} \big(\wh t'(s)+|\wh q'|(s)\big)\dd s=
    \sigma(s_1)=\tilde S.
  \end{align*}
\end{remark}

\GCOMMENT{The following remark could be dropped. ALEX agrees to dropping!}

\OLD{\begin{remark}[Different distances in the extended state space]
  \label{rem:dist-ph-sp}
  \upshape The choice $\wh t'+|\wh q'|=1$ a.e.\ in $[s_0,s_1]$ correspond to
  the canonical arclength parametrization of the curve $(\wh t,\wh q)$ in the
  extended state space
  \begin{equation}
    \label{eq:ex-ph-sp}
    \extQ:=[0,T]\ti\calQ \quad \text{with the distance}
    \quad
    d_\extQ((t_0,q_0),(t_1,q_1)):=|t_0{-}t_1|+d(q_0,q_1).
  \end{equation}
  More generally, one can choose any convex and positively $1$-homogeneous
  function (the restriction of a norm in $\R^2$) $g:[0,\infty)\times
  [0,\infty)\to [0,\infty)$ such that
  \begin{equation}
    \label{eq:RIF:11}
    g(0,\nu)=\nu,\quad
    g(\alpha,0)=\alpha,\quad
    \partial_\alpha g(\alpha,\nu)>0
    \quad\forall\, (\alpha,\nu)\in [0,\infty)\times [0,\infty),
  \end{equation}
  inducing the distance in $\R\times X$
  \begin{equation}
    \label{eq:RIF:10}
    d_g((t_0,q_0),(t_1,q_1)):=
    g(|t_1-t_0|,d(q_1,q_0)).
  \end{equation}
  An arclength reparametrization $(\tilde t,\tilde q):[0,\tilde S] \to \extQ$
  with respect to $d_g$ corresponds to the condition $g(|\tilde t'|,|\tilde
  q'|)=1$ a.e.\ in $[0,\tilde S]$.
\end{remark}}

Parametrized metric solutions admit various different but equivalent metric
characterizations (where we avoid to explicitly use the differential $\rmD_q$
of the energy).

\begin{proposition}
  \label{prop:char}
  Under the same assumptions of the previous Definition \ref{def3.2}, an
  absolutely continuous curve $(\wh{t},\wh{q}):(s_0,s_1)\to\extQ$ satisfying
  \eqref{e.DefPaMeSln_a}  and \eqref{e.DefPaMeSln_b}  is a
  parametrized metric solution of $(\calQ,d,\calE )$ if and only if it
  satisfies one of the following conditions (equivalent to
  \eqref{e.DefPaMeSln_c}):
  \begin{enumerate}[\bf A)]
  \item For all $s_0\le \sigma_0<\sigma_1\le s_1$ we have
    \begin{equation}
  \label{e.DefPaMeSln_d}
  \begin{aligned}
    \calE (\wh{t}(\sigma_1),\wh{q}(\sigma_1))&-
    \calE (\wh{t}(\sigma_0),\wh{q}(\sigma_0))-
    \int_{\sigma_0}^{\sigma_1} \pl_t\calE (\wh{t}(s),\wh{q}(s))\,\wh{t}'(s)
    \,\dd s
    \\&
    \le -\int_{\sigma_0}^{\sigma_1} M \big(\wh{t}'(s),|\wh{q}'|(s),
    \Slope\calE {\wh{t}(s)}{\wh{q}(s)}\big)\,\dd s.
   \end{aligned}
 \end{equation}
 \item  Eqn.\ \eqref{e.DefPaMeSln_d}
   holds just for $\sigma_0=s_0$ and $\sigma_1=s_1$, i.e.
   \begin{equation}
  \label{e.DefPaMeSln_e}
  \begin{aligned}
    \calE (\wh{t}(s_1),\wh{q}(s_1))&-
    \calE (\wh{t}(s_0),\wh{q}(s_0))-
    \int_{s_0}^{s_1} \pl_t\calE (\wh{t}(s),\wh{q}(s))\,\wh{t}'(s)
    \,\dd s
    \\&
    \le -\int_{s_0}^{s_1} M \big(\wh{t}'(s),|\wh{q}'|(s),
    \Slope\calE {\wh{t}(s)}{\wh{q}(s)}\big)\,\dd s.
   \end{aligned}
 \end{equation}
  \item
    For $ a.a.\ s\in(s_0,s_1)$ we have
    \begin{equation}
      \label{e:conclu1}
      \frac{\rmd}{\rmd s}\calE (\wh{t}(s),\wh{q}(s))-
      \pl_t\calE (\wh{t}(s),\wh{q}(s))\,\wh{t}'(s)
      % \langle -\rmD_q\calE (t(s),q(s)),q'(s)\rangle
      =
       -  |\wh{q}'|(s)\,
      \Slope\calE{\wh{t}(s)}{\wh{q}(s)}\,,
    \end{equation}
    and one of the following (equivalent) properties
    \begin{subequations}
      \begin{gather}
        \label{eq:equivalent1}
        M \big(\wh{t}'(s),|\wh{q}'|(s), \Slope\calE
        {\wh{t}(s)}{\wh{q}(s)}\big) = |\wh{q}'|(s)\,
        \Slope\calE{\wh{t}(s)}{\wh{q}(s)}\,,\\[0.2em]
        \label{eq:equivalent2}
        \big(\wh{t}'(s),|\wh{q}'|(s),
        \Slope\calE{\wh{t}(s)}{\wh{q}(s)}\big) \,\in \, \Xi \,,
        \\[0.2em]
        \label{eq:equivalent3}
        \left\{ \ba{rcl}
          \wh{t}'(s)>0&\Longrightarrow&
          \Slope\calE{\wh{t}(s)}{\wh{q}(s)}\leq 1,\\
          |\wh{q}'|(s)>0&\Longrightarrow&
          \Slope\calE{\wh{t}(s)}{\wh{q}(s)}\geq 1. \ea \right.
      \end{gather}
    \end{subequations}
  \end{enumerate}
  In particular, by \eqref{e:conclu1} and \eqref{eq:equivalent1},
   the following \emph{identity} holds a.e.\ in $(s_0,s_1)$:
    \begin{equation}
      \label{e.DissipationId}
      %\begin{aligned}
        \frac{\rmd}{\rmd s}\calE (\wh{t}(s),\wh{q}(s))-
        \pl_t\calE (\wh{t}(s),\wh{q}(s))\,\wh{t}'(s)
        %\\&
        =- M \big(\wh{t}'(s),|\wh{q}'|(s),
        \Slope\calE {\wh{t}(s)}{\wh{q}(s)}\big).
        %\ \text{for a.a.\    }s\in(s_0,s_1)\,,
      %\end{aligned}
    \end{equation}
\end{proposition}
\begin{proof}
  \textbf{A)} is just the integral formulation of \eqref{e.DefPaMeSln_c}.

  We note that the chain rule  \RRSS inequality \eqref{e3.6},  combined
  with \eqref{e.DefPaMeSln_c}  and \eqref{Mprop3}, \RREE implies \eqref{e:conclu1} and
  \eqref{eq:equivalent1}.  By condition \eqref{Mprop4}, \eqref{eq:equivalent1}
  is equivalent to \eqref{eq:equivalent2}.

  Since the set $\Xi$ can be easily characterized via
\[
(\alpha,\nu,\xi)\in
\Xi\quad\Longleftrightarrow\quad\Big(\,(\alpha>0\Rightarrow\xi\leq 1)\text{
  and }(\nu>0\Rightarrow\xi\geq 1)\, \Big) \,,
\]
we ultimately find that \eqref{eq:equivalent2} can be replaced by the simple
relations \eqref{eq:equivalent3}.

Having obtained the equivalence between \eqref{e.DefPaMeSln_c} and
\textbf{C)}, we can now show that \textbf{B)} is sufficient to
characterize parametrized metric solutions (the necessity is
trivial): \RRSS again \RREE applying the chain rule \RRSS
\eqref{classical-ch-rule}-\eqref{e3.6}, \RREE we get
\begin{displaymath}
  \int_{s_0}^{s_1} \Big( M \big(\wh{t}'(s),|\wh{q}'|(s),
    \Slope\calE {\wh{t}(s)}{\wh{q}(s)}\big)-
    \langle -\rmD_q\calE(\wh t(s),{\wh q(s)}),\wh q'(s)\rangle\Big)\,\dd s\le0
\end{displaymath}
so that \eqref{Mprop3} and \eqref{useful-later} yield
\begin{displaymath}
  M \big(\wh{t}'(s),|\wh{q}'|(s),
  \Slope\calE {\wh{t}(s)}{\wh{q}(s)}\big)=
  \langle -\rmD_q\calE(\wh t(s),{\wh q(s)}),\wh q'(s)\rangle=
  |\wh{q}'|(s)\,
      \Slope\calE{\wh{t}(s)}{\wh{q}(s)}
\end{displaymath}
for a.a.\ $s\in (s_0,s_1)$. We thus get \eqref{e:conclu1} and
\eqref{eq:equivalent1}.
\end{proof}

\begin{remark}[Mechanical interpretation]
\label{rem:mechanical} \upshape The evolution described by
  relations \eqref{eq:equivalent3} bears the following mechanical interpretation, cf.\
  \cite{EfeMie06RILS}. \RRSS Indeed, \RREE with $(\alpha,\nu,\xi)=(\wh{t}',|\wh{q}'|,
\Slope\calE{\wh{t}}{\wh{q}})$ we can use the decomposition $
\Xi=\Xi^{\text{stick}}\cup \Xi^{\text{slip}} \cup
\Xi^{\text{jump}}$:
\begin{itemize}
\item  $(\wh{t}'>0, \  |\wh{q}'|=0)$ leads to \emph{sticking}
  ($(\alpha,\nu,\xi) \in \Xi^{\text{stick}}$),
\item $(\wh{t}'>0, \, |\wh{q}'|>0)$ leads to
  \emph{rate-independent evolution} ($(\alpha,\nu,\xi) \in \Xi^{\text{slip}}$),
\item when $(\wh{t}'=0, \, |\wh{q}'|>0)$, the system has switched to a
  \emph{viscous regime}, which is seen as a jump in the (slow) external time
  scale (the time function $t$ is frozen and $(\alpha,\nu,\xi) \in \Xi^{\text{jump}}$).
\end{itemize}
\end{remark}

\begin{remark} \label{rem:bipotentials} \upshape
Properties \eqref{Mprop3} and \eqref{Mprop4} seem to be related to the notion
of \emph{bipotential} (cf.\ e.g.\ \cite{BuDeVa08ECBG}), which was proposed
for studying non-associated constitutive laws in mechanics by convex
analysis tools. We recall that, given two (topological, locally
convex) spaces in duality $Z$ and $Z'$, a function $b: Z \times Z'
\to (-\infty,\infty]$ is called a bipotential if it is convex,
lower
semicontinuous with respect to both arguments, and fulfills for all $(\nu,\xi)
\in Z \times Z'$
\[
b(\nu,\xi) \geq \langle \xi,\nu\rangle  \ \text{ and } \ \Big(\xi\in
\partial_{\nu} b(\cdot,\xi)(\nu) \ \Leftrightarrow \ \nu\in
\partial_{\xi}
b(\nu,\cdot)(\xi) \ \Leftrightarrow \ b(\nu,\xi) =\langle
\xi,\nu\rangle \Big) \,.
\]
Indeed, for every $\alpha>0 $  the functions $M_\eps
(\alpha,\cdot,\cdot)$ given by \eqref{e3.11+1} and $M_0
(\alpha,\cdot,\cdot)$ by \eqref{ex_1_M_0} are bipotentials on
$[0,\infty) \times [0,\infty)$.
\end{remark}

The next result ensures that the
abstract metric evolution formulation developed here reduces to the
one stated in Proposition~\ref{p:2.1}\COMMENT{, at the end of the previous
section}.

\begin{proposition}\label{p3.equiv.evol}
  Let $(\calQ,d,\calE )$ satisfy \eqref{assq}, \eqref{asserre},
  and \eqref{assene}. Then, a curve $(\wh{t},\wh{q}) \in \AC ([s_0,s_1];
  \extQ)$ is a parametrized metric solution of the
  rate-independent system $(\calQ,d,\calE )$ if and only if there exists
  $\lambda:(s_0,s_1)\to  [1  ,\infty)$ such that \eqref{e2.6} holds.
\end{proposition}
\begin{proof}
By \eqref{newform1}, condition \eqref{e.DefPaMeSln_b} in
Definition~\ref{def3.2} coincides with  the third of \eqref{e2.6}.
Now,  let us first suppose that \eqref{e.DefPaMeSln_c}
holds, and  set $\lambda(s):= \max\{\Slope\calE{\wh{t}(s)}{\wh{q}(s)},\,
1\}\,$. We shall prove that the triple $(\wh{t},\wh{q},\lambda)$
fulfills \eqref{e2.6} on $(s_0,s_1)$.
 Indeed, \eqref{e:conclu1},
 \eqref{newform2}, and \eqref{duality_op:3} yield
\begin{equation}
\label{subdif-rel} -\rmD_q\calE (\wh{t}(s),\wh{q}(s)) \in
\Slope\calE{\wh{t}(s)}{\wh{q}(s)} \,\pl \calR_1
(\wh{q}(s),\wh{q}'(s))\quad \foraa\, s\in (s_0,s_1)\,.
\end{equation}
Now, let us fix $\bar{s} \in (s_0,s_1)$ at which \eqref{subdif-rel}
holds: if $|\wh{q}'|(\bar{s}) >0$, taking into account the second of
\eqref{eq:equivalent3} we find that
$\lambda(\bar{s})=\Slope\calE{\wh{t}(\bar{s})}{\wh{q}(\bar{s})} $
and that, by \eqref{subdif-rel}, the triple $(\wh{t}',
|\wh{q}'|,\lambda) $ satisfies the first of \eqref{e2.6} at
$s=\bar{s}$. On the other hand, if $|\wh{q}'|(\bar{s}) =0$,
necessarily $\wh{t}'(\bar{s})>0 $ \RRSS  by \eqref{e.DefPaMeSln_b}
\RREE  and the first of \eqref{eq:equivalent3} gives that
$\Slope\calE{\wh{t}(\bar{s})}{\wh{q}(\bar{s})} \leq 1$. In this
case, $\lambda(\bar{s})=1$ and \eqref{subdif-rel} implies
\[
-\rmD_q\calE (\wh{t}(\bar{s}),\wh{q}(\bar{s})) \in \pl
\calR_1 (\wh{q}(\bar{s}),0)\,,
\]
hence we again conclude that $(\wh{t}'(\bar{s}),
|\wh{q}'|(\bar{s}),\lambda(\bar{s})) $ fulfills~$\eqref{e2.6}_1$.

Conversely, from the first of \eqref{e2.6} we read that for a.a.\
$s\in (s_0,s_1)$
\begin{gather}
  \label{G1}
  \Slope\calE{\wh t(s)}{\wh q(s)}\, |\wh{q}'|= \langle   -\rmD_q\calE
  (\wh{t}(s),\wh{q}(s)),\wh{q}'(s) \rangle,\\
  \label{G2}
  \Slope\calE{\wh t(s)}{\wh q(s)}\le \lambda(s),\\
  \label{G3}
  \Slope\calE{\wh t(s)}{\wh q(s)}= \lambda(s)\quad\text{if }|\wh{q}'|>0.
 \end{gather}
Then \eqref{e.DefPaMeSln_c} follows, if we check
\eqref{eq:equivalent3}. Indeed, if $\wh t'>0$, then \RRSS the second
of \RREE \eqref{e2.6} yields $\lambda=1$. Therefore the first
condition of \eqref{eq:equivalent3} follows from \eqref{G2}. If
$|\wh q'|>0$, combining \eqref{G3} and the constraint $\lambda\ge1$
of \eqref{e2.6}, we  also get  the second  of
\eqref{eq:equivalent3}.
%
% second of Choosing $v =0$ gives $\lambda \calR_1 (\wh{q},\wh{q}') \leq
% \langle -\rmD_q\calE (\wh{t},\wh{q}),\wh{q}' \rangle$ a.e.\ in
% $(s_0,s_1)$, while $v =w+\wh{q}'$, with $w $ arbitrarily fixed in
% $\rmT_{\wh{q}}\calQ $, yields
%  $\lambda \calR_1 (\wh{q},w) \geq \langle
% -\rmD_q\calE (\wh{t},\wh{q}),w \rangle$. In particular, in view
% of \eqref{newform2} we find $ \lambda(s)
% =\Slope\calE{\wh{t}(s)}{\wh{q}(s)}$ for a.a.\ $s \in
% (s_0,s_1)$ such that $|\wh{q}'|(s)>0$. Then, \eqref{equiv2}
% holds. On the other hand,
%  also taking into account the
% second of \eqref{e2.6}, we deduce that $(\wh{t}',|\wh{q}'|,
% \Slope\calE {\wh{t}}{\wh{q}} \in \Xi$, and conclude the
% proof of \ITEM{(iii)}.
 \end{proof}

\paragraph{Convergence of the viscous approximation.}
The main result of this section states that limits $(\wh{t},\wh{q})$
of parametrized solutions $(\wh{t}_\eps,\wh{q}_\eps)$ of the viscous
system \eqref{e3.10}, with $\psi=\psi_\eps$, are actually
parametrized metric solutions of the rate-independent system \RRSS
$(\calQ,d,\calE)$. \RREE   By the standard energy estimates and an
elementary rescaling, it is not restrictive to assume that the
domain of $(\wh{t}_\eps,\wh{q}_\eps)$ is a fixed interval
$[s_0,s_1]$, independent of $\eps$.

\begin{theorem}[Vanishing viscosity limit]
\label{3.3}
Let $(\calQ,d,\calE)$ satisfy  \eqref{assq}, \eqref{asserre}, and
\eqref{assene}. For every $\eps >0$  let $q_\eps\in\AC([0,T]; \calQ)$ be a
solution to \eqref{e3.9} for $\psi=\psi_\eps$.
Choose nondecreasing surjective
parametrizations $\wh{t}_\eps \in \AC([s_0,s_1];[t_{0,\eps},T]) $, and let
$\wh{q}_\eps(s)=q_\eps(\wh{t}_\eps (s))$.
Suppose that there exists $q_0 \in \calQ$,
and $m\in\rmL^1((0,S))$
such that
\begin{align}
& \label{a0}
t_{0,\eps}=\wh{t}_\eps(s_0) \to 0\,,
\quad \wh{q}_\eps (s_0)=q_\eps(t_{0,\eps})
\to q_0 \ \ \text{as $\eps \searrow 0$,}
\\
&
\label{a1}  m_\eps:=\wh{t}'_\eps+|\wh{q}'_\eps|\weakto
 m\quad\text{in } \rmL^1(s_0,s_1) \ \ \text{as $\eps \searrow 0$.}
\end{align}
Then, there exist a subsequence
$((\wh{t}_{\eps_k},\wh{q}_{\eps_k}))_{k\in \N}$ with $\eps_k\searrow
0$  and $(\wh{t},\wh{q})\in \AC([s_0,s_1]; \extQ)$ such that
$(\wh{t}(s_0),\wh{q}(s_0))= (0,q_0)$, and,  as $k\to\infty$,
\begin{gather}
\label{conv1}( \wh{t}_{\eps_k}, \wh{q}_{\eps_k})  \to (\wh{t} , \wh q)\
  \text{ in } \rmC^0([s_0,s_1];\extQ)),  \\
  %\text{ for } k\to \infty,\\
\label{conv2} ( \wh{t}_{\eps_k}', |\wh{q}_{\eps_k}'|)  \weakto
(\wh{t}' , |\wh q'|)\
  \text{ in } \rmL^1([s_0,s_1];\R^2),\qquad
\int_{t_{0,\eps_k}}^T |q_{\eps_k}'|(t)\dd t\to
\int_{s_0}^{s_1}|\wh q'|(s)\d s.
\end{gather}
The limit $(\wh{t},\wh{q})$  is a degenerate parametrized metric solution
of $(\calQ,d,\xi)$ (i.e.\ it satisfies \eqref{e.DefPaMeSln_a} and
\eqref{e.DefPaMeSln_c}),  and it is
nondegenerate (recall \eqref{e.DefPaMeSln_b})
if $m(s)>0$ a.e.\ in $(s_0,s_1)$.
\end{theorem}
\begin{proof}
Eqns.\ \eqref{e:mder} and \eqref{a1} yield
%for all $s_0\le r_0<r_1\le s_1$ and $\eps>0$
\begin{equation}
  \label{est1}
d(\wh{q}_\eps (r_0), \wh{q}_\eps (r_1)) \leq \int_{r_0}^{r_1} |\wh{q}'_\eps|(s) \,
\rmd s \leq \int_{s_0}^{r_1} m_\eps(s) \, \rmd s.
\end{equation}
In particular, choosing  $r_0=s_0$ and using  \RRSS
\eqref{a0}-\eqref{a1} \RREE  we find $C>0$ such that
\begin{equation}
  \label{esteem1} d(q_0, \wh{q}_\eps (t)) \leq C \quad\text{for all } t \in
  [s_0,s_1] \text{ and all }\eps>0\,.
\end{equation}
Moreover, it follows from \eqref{a1} that the sequences $\{ \wh{t}'_\eps\}$
and $\{ |\wh{q}'_\eps| \}$ are bounded and uniformly integrable in $\rmL^1
(s_0,s_1)$.  Hence, on the one hand, the Ascoli-Arzel\`{a} compactness theorem
and its version for metric spaces \cite[Prop.\,3.3.1]{AmGiSa05GFMS} yield that
there exists an absolutely continuous curve $(\wh{t},\wh{q}): [s_0,s_1] \to
\extQ$ such that, up to a subsequence, convergences \eqref{conv1}
hold. On the other hand, by the Dunford-Pettis criterion (see, e.g.,
\cite[Cor.\,IV.8.11]{DunSch58LOPI}), there exists
$\eta\in L^1
(s_0,s_1)$ such that, up to the extraction of a (not
relabeled) subsequence,
\begin{equation}
  \label{e:conv0}
  {\wh{t}_\eps}' \weakto \wh t',
  \qquad{|\wh{q}_\eps'|} \weakto \eta \quad \text{in } \rmL^1 (s_0,s_1) \
    \text{ as } \eps \searrow 0.
\end{equation}
Passing to the limit in \eqref{est1} we easily get
\begin{equation}
\label{inequality} |{\wh{q}}'|(s) \leq \eta(s) \leq m(s)
 \qquad \foraa\, s\in (s_0,s_1)\,.
\end{equation}

Now, the smoothness of $\calE$ and assumption \eqref{assq} yield
that $\calE$, $\partial_t \calE$, and $|\partial_q \calE|$ are
continuous with respect to both arguments, \RRSS thus \RREE  we
readily infer from \eqref{conv1} that
\begin{equation}
\label{e:conv2}
\begin{gathered}
 \calE({\wh{t}}_\eps, {\wh{q}}_\eps) \to
 \calE({\wh{t}}, {\wh{q}}), \ \  \Slope\calE{{\wh{t}}_\eps}{{\wh{q}}_\eps} \to
\Slope\calE {{\wh{t}}}{\wh{q}} \ \ \text{and} \ \ \pl_t\calE ({\wh{t}}_\eps,
{\wh{q}}_\eps) \to \pl_t\calE ({\wh{t}}, {\wh{q}})
\\
\text{uniformly in  }[s_0,s_1] \text{ for } \eps \searrow 0\,.
\end{gathered}
\end{equation}
To proceed further, we integrate \eqref{e3.11} over $[s_0,s_1]$ and
obtain
\begin{equation}
\label{integral-form}
\begin{aligned} &
\calE (\wh{t}_\eps(s_0),\wh{q}_\eps(s_0))
-  \calE (\wh{t}_\eps(s_1),\wh{q}_\eps(s_1))
+ \int_{s_0}^{s_1}
\pl_t\calE (\wh{t}_\eps(s),\wh{q}_\eps(s))\,\wh{t}'_\eps(s) \dd s\\
& \geq
\int_{s_0}^{s_1}
M_\eps\big(\wh{t}'_\eps(s),|\wh{q}'_\eps|(s),
\Slope\calE{\wh{t}_\eps(s)}{\wh{q}_\eps(s)}\big) \dd s  \,.
\end{aligned}
\end{equation}
On the left-hand side we can pass to the limit $\eps \to 0$ \RRSS
using \eqref{e:conv0}  and \eqref{e:conv2}, \RREE whereas for the
right-hand side we use Part (B) of Lemma \ref{l3.M-Gamma}:
\[
\begin{aligned}
& \calE ({\wh{t}}(s_0),{\wh{q}}(s_0)) -
\calE ({\wh{t}}(s_1),{\wh{q}}(s_1)) +
\int_{s_0}^{s_1}
\pl_t\calE ({\wh{t}}(r),{\wh{q}}(r))\,{\wh{t}}'(r)\, \rmd r
\\
& = \lim_{\eps \searrow 0} \Big(
\calE ({\wh{t}}_\eps(s_0),{\wh{q}}_\eps(s_0)) -
\calE ({\wh{t}}_\eps(s_1),{\wh{q}}_\eps(s_1)) +
\int_{s_0}^{s_1}
\pl_t\calE ({\wh{t}}_\eps(r),{\wh{q}}_\eps(r))\,{\wh{t}}_\eps'(r)\,
\rmd r \Big)
\\
& \geq  \liminf_{\eps \searrow 0} \int_{s_0}^{s_1}
M_\eps\big({\wh{t}}_\eps'(r),|{\wh{q}}_\eps'|(r),
\Slope\calE{{\wh{t}}_\eps(r)}{{\wh{q}}_\eps(r)}\big)\,
\rmd r
%\\
%& \geq \int_{s_0}^{s_1} M_0\big({\wh{t}}'(r),\eta(r),
%\omega(r)\big)\, \rmd r
\\
& \geq \int_{s_0}^{s_1}
M_0\big({\wh{t}}'(r), \eta(r),
\Slope\calE{{\wh{t}}(r)}{{\wh{q}}(r)}\big)\, \rmd r
\\
& \geq  \int_{s_0}^{s_1}
M_0\big({\wh{t}}'(r),|{\wh{q}}'|(r),
\Slope\calE{{\wh{t}}(r)}{{\wh{q}}(r)}\big)\, \rmd r
\end{aligned}
\]
where we have used \eqref{inequality} and the monotonicity of
$M_0(\alpha,\cdot,\xi)$ for the last estimate.

We see that $({\wh{t}},{\wh{q}})$ fulfills \eqref{e.DefPaMeSln_e} a.e.\
in~$(s_0,s_1)$ , and it is therefore a (possibly) degenerate parametrized
metric solution of $(\calQ,d,\calE)$. Moreover, comparing the last
inequalities with the integrated form of \eqref{e.DissipationId}, we get
 \begin{displaymath}
   M_0\big({\wh{t}}'(r),\eta(r),
   \Slope\calE{{\wh{t}}(r)}{{\wh{q}}(r)}\big)=
   M_0\big({\wh{t}}'(r),|{\wh{q}}'|(r),
\Slope\calE{{\wh{t}}(r)}{{\wh{q}}(r)}\big)<\infty
\end{displaymath}
for a.a.\ $r\in (s_0,s_1)$. Since $M_0(\alpha,\cdot,\xi)$ is strictly monotone
in its domain of finiteness, we get $\eta(r)=|\wh q'|(r)$ for a.a.\ $r\in
(s_0,s_1)$, thus obtaining \eqref{conv2}.  Using the first convergence in
\eqref{e:conv0} we also find $\wh t'+|{\wh{q}}'|= m$ and the last assertion
follows.
\end{proof}

\begin{remark}[Preservation of arclength parametrizations]
  \label{rem:pres}
  \upshape
  If $(\wh t_\eps,\wh q_\eps)$ are arclength parametrization
  (i.e.\ $m_\eps\equiv 1$), then  their limit $(\wh t,\wh q)$
  still satisfies the arclength property $\wh t'+|\wh q'|=1$,
  thanks to \eqref{conv2}.  This generalizes
  \cite[Cor.\:3.6]{EfeMie06RILS}.
\end{remark}

\begin{remark}
  \label{rem:point-forward-stability} \upshape Mimicking the argument of the
  proof of Proposition~\ref{3.3}, under the same assumptions it is also
  possible to prove a result of stability with respect to initial data for
  parametrized metric solutions. Namely, let $\{ (\wh{t}_n, \wh{q}_n)\}$ be a
  sequence of parametrized metric solutions on a time interval $[s_0,s_1]$,
  such that $(\wh{t}_n(s_0), \wh{q}_n (s_0))=(t_0^n, q_0^n) $ for every $n \in
  \N$, with $(t_0^n, q_0^n) \to (t_0,q_0)$ and $m_n:=\wh t_n'+|\wh
  q_n'|\weakto m$ in $L^1(s_0,s_1)$.  Then, there exists a parametrized metric
  solution $(\wh{t}_\infty, \wh{q}_\infty)$, starting from $(t_0,q_0)$, such
  that, up to the extraction of a subsequence, $(\wh{t}_n, \wh{q}_n) \to
  (\wh{t}_\infty, \wh{q}_\infty)$ uniformly in $[s_0,s_1]$, with $(\wh{t}_n',
  |\wh{q}_n'|) \weakto (\wh{t}_\infty', |\wh{q}_\infty'|)$ in
  $\rmL^1(s_0,s_1)$.  In fact, in \cite{MiRoSa08?VVLM} we shall prove the
  above result, as well as the vanishing viscosity analysis of
  Theorem~\ref{3.3},  in the more general setting detailed in
  Section~\ref{s:outlook}.
\end{remark}

\section{BV solutions} \label{ss:BB.sol}

Before introducing the notion of BV solution to the rate-independent system
driven by $\calE$, we recall some definitions and properties of BV
functions on $[0,T]$ with values in the space $(\calQ,d)$ introduced in the
previous section. Note that, however, the following notions are indeed
independent of the Finsler setting \eqref{assq}--\eqref{asserre} and can be
given for a general complete metric space.

\paragraph{Preliminaries on BV functions.}
Given a function $q:[0,T]\to \calQ$ and an interval $I\subset[0,T]$,
we define its variation on  $I$ by
\begin{equation}
\label{e:4.1}
\begin{gathered} \VAR(q,I):=\sup\sum^n_{j=1}d(q(\tau_{j-1}),q(\tau_j)),
\end{gathered}
\end{equation}
where $\sup$ is taken over all $n\in\N$ and all partitions
$\tau_0<\tau_1\cdots<\tau_{n-1}<\tau_n$ with $\tau_0,\tau_n\in I$.
We set
\[
\BV([0,T];\calQ)=\set{q:[0,T]\to \calQ}{\VAR(q,[0,T])<\infty},
\]
where we emphasize that functions are defined everywhere, as is common for
rate-independent processes.
For $q\in\BV([0,T];\calQ)$ and $t\in[0,T]$ the left and
right limits exist:
\[
q(t^-):=\lim_{h\searrow 0}q(t{-}h)\quad\text{and}\quad
q(t^+):=\lim_{h\searrow 0} q(t{+}h),
\]
where we put $q(0^-)=q(0)$ and $q(T^+)=q(T)$. In general,
the three values $q(t^-), q(t),$ and $q(t^+)$ may differ. We define
the continuity set $\rmC_q$ and the jump set $\rmJ_q$~by
\[
\ba{l} \rmC_q=\set{t\in[0,T]}{q(t^-)=q(t)=q(t^+)}, \qquad
\rmJ_q=[0,T]\setminus \rmC_q. \ea
\]
Indeed, our definition of ``Var'' is such that we have for all
$0\leq r\leq s\leq t\leq T$
\begin{equation}
\label{e:var1} \VAR(q,[r,s])=d(q(r),q(r+))+ \VAR(q,(r,s))+d(q(s-),q(s))\,,
\end{equation}
and the additivity property
\begin{equation}
\label{additiv} \VAR(q,[r,t])=
\VAR(q,[r,s])+\VAR(q,[s,t]).
\end{equation}
When calculating the variation of $q$ over an interval $I$,  one has
to be careful \RRSS with \RREE  (possible) jumps at the boundary of
$I$, if $I$ contains  boundary points. Now,  for a function
$q\in\BV([0,T],\calQ)$ we introduce the nondecreasing function
\[
V_q:[0,T]\to[0,\infty)\,, \qquad  V_q (t):= \VAR(q,[0,t]).
\]
The distributional derivative of $V_q$ defines a nonnegative Radon
measure $\mu_q$ such that
 \begin{equation}\label{e4.2}
  \mu_q([s,t])= V_q(t)-V_q(s)
\qquad \forall\, t,\, s \in \rmC_q\,,
\end{equation}
and, more generally (see \cite[2.5.17]{Fede69GMT})
\begin{equation}
  \label{e.Stieltjes}
  \int_0^T \zeta(t)\,\dd V_q(t)=\int_0^T \zeta(t)\, \mu_q(\dd t)
  \quad\text{for all } \zeta\in \rmC^0_\rmc(0,T),
\end{equation}
where $\int_0^T \zeta \dd V_q$ denotes the Riemann-Stieltjes integral.
 As usual, $\mu_q$ can be
decomposed into a \emph{continuous} (also called \emph{diffuse})
part $\mu_q^\cont$ and a discrete part $\mu_q^\discr $, where for
\RRSS  a Borel set \RREE  $A\subset[0,T]$ we have
\begin{equation}\label{e4.3}
\mu_q^\discr (A)=\mu_q(A\cap \rmJ_q)=\sum_{t\in A \cap \rmJ_q}d(q(t^-), q(t))+d(q(t),q(t^+))\,,
\end{equation}
in accordance with formula \eqref{e:var1} above.

In the following technical lemma (whose proof is postponed to the
end of this section), we will discuss the link between a BV map
$q:[0,T]\to \calQ$ and its graph $\gph(t)=(t,q(t))$ in the extended
state space $\extQ$, endowed with the distance $
d_\extQ((t_0,q_0),(t_1,q_1)):=|t_0{-}t_1|+d(q_0,q_1)$ \GCOMMENT{(see
Remark \ref{rem:dist-ph-sp})}.  We denote by $\calL_{[a,b]}$ the
Lebesgue measure on the interval $[a,b]$,  whereas $\calL$
denotes a general one-dimensional Lebesgue \RRSS measure.

\begin{lemma}
  \label{le:chain}
  Let $q\in \BV([0,T];\calQ)$ and $\gph \in \BV([0,T];\extQ)$ with
  $\gph (t):=(t,q(t))$. Set
  %$r(t):=V_q(t)=  \VAR(q;[0,t])$ and
  \begin{displaymath}
    \rho(t):=V_q(t)=\VAR(q,[0,t]),\ R:=\rho(T);\quad
    \sigma (t):=V_\gph (t)=\VAR(\gph,[0,t]),\ S:=\sigma(T),
  \end{displaymath}
  with their right-continuous inverse
  \begin{displaymath}
    \wh \tau(r):= \sup\set{t\in [0,T]}{V_q(t)=\rho(t)<r},\quad
    \wh t(s):=\sup\set{t\in [0,T]}{V_\gph(t)=\sigma(t)<s}.
  \end{displaymath}
 Then, the following statements hold:
  \begin{enumerate}[\bf A)]
  \item $\sigma (t)=t{+}\rho(t), \ \ \rmJ_\gph =\rmJ_q,\ \
    \rmC_\gph =\rmC_q, \ \ \mu_\gph =\calL+\mu_q, \ \
    \mu^\cont_\gph =\calL+\mu^\cont_q, \ \ \mu_\gph ^\discr=\mu_q^\discr$.
  \item There exist $1$-Lipschitz maps
    $\wh \gph =(\wh t,\wh q):[0,S]\to \extQ$ and
    $\wh\rho:[0,S]\to [0,R]$ such that
    $\gph (t)=\wh \gph (\sigma (t))$, $\rho(t)=\wh\rho (\sigma (t))$ for all
    $t\in [0,T]$. The map $\wh t$ is uniquely determined,
    it is  \RRSS the right-continuous \RREE  inverse of $\sigma $, and it is injective on
    $\wh \rmC_q:=\sigma (\rmC_q)=\wh t^{-1}(\rmC_q)$.
    The maps $\wh \gph$
    and $\wh\rho $ are uniquely determined on the set
    $\wh\rmC_q$ and satisfy $\wh q(s)=q(\wh t(s))$ and
    $\wh\tau(\wh\rho (s))=\wh t(s)$
    for all $s\in \wh\rmC_q$.
  \item $\wh t_\# (\calL_{[0,S]})=\mu_\gph $
    and $\wh\tau_\#(\calL_{[0,R]})=\mu_q$, in the sense that
    for all bounded Borel function $\zeta:[0,T]\to \R$ and
    Borel set $A\subset [0,T]$
    \begin{equation}
      \label{e.change_of_var}
      \int_{\wh t^{-1}(A)} \zeta(\wh t(s))\,\dd s=
      \int_A \zeta(t)\,\mu_\gph (\dd t),\quad
      \int_{\wh\tau^{-1}(A)} \zeta(\wh\tau(r))\dd r=
      \int_A \zeta(t)\,\mu_q(\rmd t).
    \end{equation}
    In particular, if  $A\subset [0,T], B\subset \rmC_q\subset [0,S]$
    are Borel sets, then
    \begin{equation}
      \label{e.null1}
      \mu_\gph (A)=\calL(\wh t^{-1}(A)),\qquad
      \calL(B)=\mu_\gph (\wh t(B)).
    \end{equation}
  \item The Lebesgue densities of the measures
    $\calL,\mu_q^\cont \ll\mu_\gph ^\cont$
    with respect to $\mu_\gph ^\cont$ are expressed by the formulae
    \begin{equation}
      \label{e.densities}
      \frac {\dd\calL}{\dd\mu^\cont_\gph }=\wh t'\circ \sigma ,\quad
      \frac {\dd\mu_q^\cont}{\dd\mu^\cont_\gph }=|\wh q'|\circ \sigma =
      \wh\rho'\circ \sigma .
    \end{equation}
  \end{enumerate}
\end{lemma}

\paragraph{The notion of BV solution.}
%Hereafter, we shall suppose  that $\calE$ complies
%with \eqref{assene}.
% \ \\
% \GCOMMENT{Here is stressed that also $S$ is a Finsler distance.
%  Since we used a calligraphic letter for $\calR$, I adopted
%  the same convention for $\calS$
%  (there is a macro which can be modified)}
%
Let us first introduce a new family of $1$-homogeneous dissipation functionals $\DS\alpha t\cdot\cdot:\rmT\calQ\to[0,\infty)$,
depending on the two parameters $\alpha\in [0,\infty)$ and $t\in [0,T]$, defined as
\begin{equation}
  \label{eq:DS}
  \DS\alpha t q v:= \max\big\{\Slope\calE tq,\alpha\big\}\,\calR_1(q,v)
\end{equation}
Notice that for all $\alpha>0$, $t\in [0,T]$, and  $q\in \calQ$ the
functional  $\DS \alpha tq\cdot$ is a norm on $\rmT_q\calQ$ (possibly
degenerate, when $\alpha=0$), thus satisfying condition \eqref{asserre}. As in
Section \ref{sez3.1}, we can therefore consider the corresponding Finsler
distances $\dS \alpha t\cdot\cdot:\calQ\ti \calQ\to [0,\infty)$ via
\begin{equation}\label{e.4.1}
\dS \alpha t{q_0}{q_1}:=
\inf\bigset{\int^1_0 \DS\alpha t {y(s)}{y'(s)}\dd
s}{y\in\calA(q_0,q_1)} \,.
 \end{equation}
% Given $q_0,\, q_1 \in \calQ$, we set
% \[
% \calA(q_0,q_1)=\set{y\in\AC([0,1],\calQ)}{y(0)=q_0,\
%   y(1)=q_1}\,.
% \]
The functional $\distS_0$ is called \emph{slope distance}
and $\dS \alpha t\cdot\cdot$ also admits the equivalent formulation
in terms of the metric velocity
\begin{equation}\label{e.4.1bis}
\dS\alpha t{q_0}{q_1}:=
\inf\bigset{\int^1_0 \max\big\{\Slope\calE t{y(s)},\alpha\big\}
  \,|y'|(s)\dd s}{y\in\calA(q_0,q_1)} \,.
 \end{equation}
For $\alpha>0$ the infimum in \eqref{e.4.1} and \eqref{e.4.1bis} is attained.
% , since for all $t \in [0,T]$ $\distS (t,\cdot,\cdot)$ is some sort of
% Riemannian (semi-)distance on the space $\calQ$ (weighted by $\Slope\calE
% t\cdot$), in particular fulfilling $\distS (t,q_0,q_1)=\distS (t,q_1,q_0)$
% for all $(t,q_0,q_1) \in \extQ \times \calQ$.  Note that, if $\distS
% (t,q_0,q_1) <\infty$, then there exists an optimal path $y \in
% \calA(q_0,q_1)$ attaining the minimum in \eqref{e.4.1} and, using a suitable
% parametrization, we may assume $|y'|(s)=d(q_0,q_1)$ a.e.\ in $[0,1]$.

A straightforward consequence of the symmetry
$\calR_1(q,-v)=\calR_1(q,v)$ (see \eqref{asserre}) is  that $\dS
\alpha t{q_0}{q_1} = \dS \alpha t{q_1}{q_0}$. Using  the chain
rule inequality \eqref{e3.6}  \RRSS we find \RREE
\begin{equation}
  \label{S_bounds_E} |\calE(t, q_1) - \calE(t,q_0)| \leq
  \dS 0t{q_0}{q_1}\le \dS \alpha t{q_0}{q_1}
  \quad \text{for all } (t,q_0,q_1) \in [0,T] \times
  \calQ \times \calQ\,.
\end{equation}

The notion of BV solution to \RRSS the rate-independent system \RREE
$(\calQ,d,\calE)$, which we are going to introduce, relies on a
version of the chain rule for BV functions with values in a metric
space. In order to state it, for a general $q \in \BV ([0,T];\calQ)$
and $0\leq t_0 \le t_1 \leq T$ we define
\begin{equation}\label{e4.4}
\begin{aligned} \Sigma_0(q, [t_0,t_1]):=& \int_{t_0}^{t_1}
\Slope\calE r{q(r)}\, \mu_q^\cont(\rmd r)
+\dS 0{t_0}{q(t_0)}{q(t_0^+)}+\dS0{t_1}{q(t_1^-)}{q(t_1)}\\
&+\sum\limits_{t\in \rmJ_q\cap
(t_0,t_1)}[\dS0t{q(t^-)}{q(t)}
{+}\dS0t{q(t)}{q(t^+)} ]\,.
\end{aligned}
\end{equation}
Based on \eqref{S_bounds_E}, we define a second functional $\Gamma$
via
\begin{equation}\label{e4.Gamma}
\begin{aligned} \Gamma(q, [t_0,t_1]):=
& \int_{t_0}^{t_1}
 \Slope\calE r{q(r)}\, \mu_q^\cont(\rmd r)\\
&+|\calE(t_0,q(t_0)){-} \calE(t_0,q(t_0^+))|
 +|\calE(t_1,q(t_1^-)){-} \calE(t_1,q(\RRSS t_1 \RREE))|\\
&+\sum\limits_{t\in \rmJ_q\cap (t_0,t_1)}\!\!\Big[\calE(t,q(t)){-} \calE(t,q(t^+))|
 +|\calE(t,q(t^-)){-} \calE(t,q(\RRSS t \RREE))|\Big]\,.
\end{aligned}
\end{equation}
Obviously, we have $\Sigma_0(q,[s,t])\geq \Gamma(q,[s,t])\geq 0$ and
both functionals $\Sigma_0(q,\cdot)$ and $\Gamma(q,\cdot)$  fulfill
the additivity property \eqref{additiv}, when considered as
functions on intervals.

\begin{proposition}
\label{p.ChainRuleBV}
Under assumptions \eqref{assq}, \eqref{asserre}, and \eqref{assene}, the
following chain rule inequality holds for all $q\in\BV([0,T],\calQ)$ and $0\leq
t_0\leq t_1\leq T$:
\begin{equation}\label{e4.5}
\calE (t_1,q(t_1))-\calE (t_0,q(t_0))
-\int_{t_0}^{t_1}\pl_t\calE (t,q(t))\dd
t \geq -\Gamma(q,[t_0,t_1])\geq  -\Sigma_0(q,[t_0,t_1])\,.
\end{equation}
\end{proposition}
\begin{proof}
  The function \RRSS $t\mapsto E(t):=\calE(t,q(t))$ \RREE  is of bounded variation
  \RRSS on $[0,T]$ \RREE and its jump set
  is contained in $\rmJ_q$.  We denote by $\eta=\frac \rmd{\rmd t} E$ its
  distributional derivative (a bounded Radon measure on $(0,T)$) and by
  $\eta^\cont$ its diffuse part, defined as $\eta^\cont(A):=\eta(A\cap
  \rmC_q)$ \RRSS  for all Borel sets $A \subset [0,T]$. \RREE
    Thanks to \eqref{S_bounds_E},  we have \eqref{e4.5}
  if we show that
\begin{equation}
  \label{eq:2}
  \eta^\cont\ge \partial_t \calE(\cdot,q(\cdot))\calL- \Slope\calE
  \cdot{q(\cdot)}\mu_q^\cont.
\end{equation}
We introduce the maps $\sigma $ and $\wh \gph =(\wh t,\wh q)$ as in
Lemma \ref{le:chain} and we set $\wh E(s):=\calE(\wh t(s),\wh q(s))$
\RRSS for all $s \in [0,S]$, \RREE so that $E(t)=\wh E(\sigma (t))$
for all $t\in [0,T]$. Since $\wh t,\wh q$ are Lipschitz continuous
and $\calE$ is of class $\rmC^1$, the classical chain rule \RRSS
\eqref{classical-ch-rule}--\eqref{e3.6} \RREE yields
\begin{equation}
  \label{eq:3}
  \wh E'(s) \ge\partial_t\calE(\wh t(s),\wh q(s)) \wh t'(s)-
  \Slope\calE{\wh t(s)}{\wh q(s))} |\wh q'|(s)\quad
  \text{for $\calL$-a.a.\ $s\in (0,S)$}.
\end{equation}
On the other hand, since $\wh E$ is a Lipschitz map and since
$\frac {\rmd \sigma }{\rmd t}=\mu_\gph $,
the general chain rule of \cite{AmbDMa90AGCR} yields
\begin{equation}
  \label{eq:5}
  \eta^\cont=({\RRSS\wh E' \RREE}\circ \sigma) \, \mu^\cont_\gph \ge
  \Big(\partial_t\calE(t,q(t)) \wh t'\circ \sigma -
  \Slope\calE{t}{q(t)} |\wh q'|\circ \sigma \Big)\mu^\cont_\gph .
\end{equation}
Taking into account \eqref{e.densities}, we conclude (notice that
$\wh E'\circ \sigma$ is well defined $\mu_\gph^\cont$-a.e., since,
for every Lebesgue negligible set $N\subset \wh \rmC_\gph=
\sigma(\rmC_\gph)$, \eqref{e.null1} yields
$\mu^\cont_\gph(\sigma^{-1}(N))=0$).
\end{proof}
Now we are able to define the notion of BV solution. The formulation
is more complicated \RRSS  than the one defining parametrized metric
solutions, \RREE but it nicely reflects the different flow regimes
of rate-independent flow, and the jumps. A shorter but much more
implicit formulation will be given in Remark \ref{prop:BV_alter}.

\begin{definition}[BV solution]
\label{def:BVsln}
Let $(\calQ,d,\calE)$ satisfy \eqref{assq}, \eqref{asserre},
\eqref{assene}. A function $q\in\BV([0,T];\calQ)$ is called a \emph{BV
  solution} of the rate-independent system $(\calQ,d,\calE )$, if  the
following four conditions hold:
\begin{subequations}
\label{e4.6} \label{e.defBVsln}
\begin{align}
\label{e.BVen}
&\begin{aligned}
  \calE (t_1,q(t_1))-\calE (t_0,q(t_0))- & \int^{t_1}_{t_0}
  \pl_t\calE (t,q(t))\dd t \\ &\leq -\Sigma_0(q,[t_0,t_1]) \
  \text{for }0\leq t_0 < t_1\leq T; \end{aligned}
\\[0.3em]
\label{e.BVa}
& \Slope\calE t{q(t)} \leq 1\quad\text{for }t\in[0,T]\setminus \rmJ_q;
\\[0.3em]
\label{e.BVb}
&  \Slope\calE t{q(t)} \geq 1\quad\text{for }t \in \sppt (\mu_q);
\\[0.3em]
\label{e.BVc}
&\begin{aligned}
 &\text{for } t\in \rmJ_q\quad\text{there exist } y^t\in \calA(q(t^-),q(t^+))
    \text{ and } \theta^t\in[0,1]\text{ such that}\\
  &(\alpha)\quad y^t(\theta^t)=q(t),\\
  &(\beta)\quad
  \Slope\calE t{{\RRSS y^t(\theta)\RREE}}\geq 1\quad
   \text{for all }\theta\in[0,1],\\
  & \ts (\gamma)\quad \calE (t,q(t^+))-\calE (t,q(t^-))=
    -\int^1_0
    \Slope\calE t{y^t(\theta)}\,|y'|(\theta)\dd \theta.
  \end{aligned}
\end{align}
\end{subequations}
\end{definition}

Again we point out that, due to the chain rule inequality
\eqref{e4.5}, relation \RRSS \eqref{e.BVen} \RREE holds as an
equality, which is the energy balance. Using this energy identity on
the intervals $[t{-}h,t]$ and $[t,t{+}h]$ and letting $h\searrow 0$
leads to the first two of the  following jump relations, which will
be used later \RRSS (recall the definition \eqref{e.4.1} of
$\distS_1$): \RREE
\begin{equation}\label{e.BVjump}
\ba{l}
\calE(t,q(t^-)){-}\calE(t,q(t))=\dS0t{q(t^-)}{q(t)}
=\dS1t{q(t^-)}{q(t)}, \\
\calE(t,q(t)){-}\calE(t,q(t^+))=\dS0t{q(t)}{q(t^+)}
=\dS1t{q(t)}{q(t^+)}, \\
\calE(t,q(t^-)){-}\calE(t,q(t^+))=\dS0t{q(t^-)}{q(t^+)}
=\dS1t{q(t^-)}{q(t^+)},
\ea
\end{equation}
for each $t\in \rmJ_q$. The third relation follows from \eqref{e.BVc}. By the
definition of the slope distance $\distS_0$,  these jump relations already include
the existence of a connecting  gradient-flow curve $y\in
\calA(q(t^-),q(t^+))$, i.e.\ $(\alpha)$, $(\beta)$,
and $(\gamma)$ of \eqref{e.BVc} follow.

The above formulation of BV solutions looks quite lengthy compared
to the more elegant forms of gradient-like flows, which can be
characterized by one inequality, cf.\ e.g., \eqref{e3.9} or
\eqref{e.DefPaMeSln_c}. However, this formulation reflects the
mechanical interpretation of the three different flow types quite
well, namely sticking, slipping and   jumping. The following result
presents a more compact form, which is however less tractable for
further analysis.

\begin{proposition}
\label{prop:BV_alter} \RRSS In the setting of \eqref{assq},
\eqref{asserre}, \eqref{assene}, \RREE
 let
%  by introducing $\distS_* :\extQ\ti
% \calQ\to [0,\infty)$ via
% \[
% \distS_* (t,q_0,q_1):=
% \inf\set{\int^1_0\max\big\{\Slope\calE t{y(s)},1\big\}
% |y'|(s)\dd s}{y\in\calA(q_0,q_1)}
% \]
% and a
$\Sigma_1(\cdot,[t_0,t_1])$
be the functional defined on $\BV([0,T],\calQ)$
via
\[
\begin{aligned} \Sigma_1(q, [t_0,t_1]):=& \int_{t_0}^{t_1}
\max\big\{\Slope\calE t{q(t)},1\big\} \mu_q^\cont(\rmd t)
+ \int_{t_0}^{t_1}\Big(\Slope\calE t{q(t)}-1 \Big)^+ \dd t\\
&+\dS1{t_0}{q(t_0)}{q(t_0^+)}+\dS1{t_1}{q(t_1^-)}{q(t_1)}\\&
+\sum\limits_{t\in \rmJ_q\cap(t_0,t_1)} [\distS_1 (t,q(t^-),q(t))
{+}{\RRSS \distS_1(t,q(t),q(t^+))\RREE}].
\end{aligned}
\]
Then, $q\in \BV([0,T];\calQ)$ is a BV solution if and only if
\begin{equation}\label{e.BVineq}
\calE (t_1,q(t_1))-\calE (t_0,q(t_0))-\int^{t_1}_{t_0}\pl_t\calE (t,q(t))\dd t
 \leq -\Sigma_1(q,[t_0,t_1]) \text{ for } 0\leq t_0 < t_1\leq T\,.
\end{equation}
\end{proposition}
\begin{proof}
  It is clear that under conditions \eqref{e.BVa}, \eqref{e.BVb}, and
  \eqref{e.BVc} a BV solution $q$ satisfies
  \begin{equation}
    \label{e.01}
    \Sigma_0(q;[t_0,t_1])=\Sigma_1(q;[t_0,t_1])\quad
    \text{for }0\le t_0< t_1\le T,
  \end{equation}
  so that \eqref{e.BVen} yields \eqref{e.BVineq}.

  Conversely, if $q\in \BV([0,T];\calQ)$ satisfies
  \eqref{e.BVineq}, the chain rule \eqref{e4.5}
\RRSS and the inequality $\Sigma_0(\cdot ;[t_0,t_1]) \leq
\Sigma_1(\cdot;[t_0,t_1])$
   yield \RREE
  \eqref{e.BVen} and \eqref{e.01}.
  Choosing e.g. $t_0=0, t_1=T$ we get
  \begin{align*}
    0&=\int_0^T
    \Big(\max\big\{\Slope\calE t{q(t)},1\big\}-
    \Slope\calE t{q(t)}\Big)\mu_q^\cont(\rmd t)
    + \int_0^T\Big(\Slope\calE t{q(t)}-1 \Big)^+ \dd t\\
    &+\Big(\dS10{q(0)}{q(0^+)}-
    \dS00{q(0)}{q(0^+)}\Big)+\Big(\dS1T{q(T^-)}{q(T)}
    -\dS0T{q(T^-)}{q(T)}\Big)\\&
    +\sum\limits_{t\in \rmJ_q}
    \Big[\dS1t{q(t^-)}{q(t)}-
    \dS0t{q(t^-)}{q(t)}
    {+}\dS1t{q(t)}{q(t^+)}-
    \dS0t{q(t)}{q(t^+)}\Big].
  \end{align*}
  Since each addendum is nonnegative, we easily find \eqref{e.BVa}
  and \eqref{e.BVb} (recalling that $|\partial\calE|$ is continuous).
  Moreover, passing to the limit in \eqref{e.BVineq} as
  $t_0\nearrow t,\ t_1\searrow t$, with $t\in \rmJ_q$, we conclude
  \begin{displaymath}
    \calE(t,q(t^+))-\calE(t,q(t^-))\le -\dS1t{q(t^-)}{q(t)}-
    \dS1t{q(t)}{q(t^+)}.
  \end{displaymath}
  Recalling \eqref{e.4.1} and the chain rule, for every $t\in \rmJ_q$
  we find a curve $y^t$ satisfying condition \eqref{e.BVc}.
\end{proof}

\paragraph{BV and parametrized metric solutions.}
 We claim  that the notion of BV solution is essentially
the same as that of parametrized metric solution. Intuitively,
\eqref{e.BVen} corresponds to  \RRSS \eqref{e:conclu1}. \RREE
Further, \eqref{e.BVa} and \eqref{e.BVb} are the analog in the BV
setting of the first of \eqref{eq:equivalent3}, which encompasses
both sticking and rate-independent evolution (recall
Remark~\ref{rem:mechanical}). \RRSS The jumping \RREE  regime is
accounted for by condition \eqref{e.BVc}: at jump times, the system
switches to a viscous, rate-dependent behavior, following a path
described by a \emph{generalized gradient flow}, see ($\gamma$) in
\eqref{e.BVc}.

In order to formalize these considerations, we return to the trajectories in
$\extQ$.
%
%From the definition of equivalence
%between admissible curves (see Definition 3.1 (b)), we see that two
%parametrized metric solutions $(\wh{t},\wh{q})$ are equivalent if
%and only if their trajectories
%\[
%\set{(\wh{t}(s),\wh{q}(s))}{s\in(s_1,s_2)}\subset\extQ
%\]
%are the same.
Indeed, we may associate with each BV solution $q_{\BV}$ a
trajectory, by filling the jumps of the graph
$\set{(t,q_{\BV}(t))}{t\in[0,T])}$ with the curves ${\RRSS y^t
\RREE}\in\AC([0,1],\calQ)$, for $t\in J_{q_{\BV}}$. Thus, we obtain
\[
\mathcal{T}=\set{(t,q_{\BV}(t))}{t\in[0,T]} \, \cup \,
\bigcup_{t\in J{q_{\BV}}} \set{(t,y^t(\theta))}{\theta\in[0,1]}.
\]
By construction, $\mathcal{T}$ is a connected curve that has exactly
the length $\VAR(q,[0,T])+T$ if we use the extended metric
$ d_\extQ((t_0,q_0),(t_1,q_1)):=|t_0{-}t_1|+d(q_0,q_1) $ on
$\extQ$. Hence, there exists an absolutely continuous
parametrization of $\mathcal{T}$, and it can be shown that this
parametrized curve is a parametrized metric solution. Indeed, in
Example~\ref{ex:5.4} we shall show, that to a given BV solution,
there may correspond infinitely many distinct parametrized metric
solutions.

On the other hand, we can pass from parametrized metric solutions
$(\wh t,\wh q)$ defined in $[s_0,s_1]$ to
BV solutions by choosing
\begin{equation}
  \label{eq:7}
  \sigma (t)\in \set{s\in [s_0,s_1]}{\wh t(s)=t }
  \quad\text{and defining}\quad
  q(t):=\wh{q}(\sigma (t))\,.
\end{equation}
Hence, $\rmJ_q=\set{t\in
[0,T]}{\sigma (t^+)>\sigma (t^-)}$, and we see that
$q(t)$ is uniquely determined
from $\wh{q}$ for $t\in[0,T]\setminus \rmJ_q$. At the jump times $t$ we
can in fact choose any point $q(t)=\wh{q}(s)$ with
$s\in[\sigma (t^-),\sigma (t^+)]$.
Note that
\begin{equation}
  \label{eq:1}
  y^t(\theta):=\wh{q}\big({\RRSS  \sigma(t^-) \RREE}+\theta
  [\sigma (t^+){-}\sigma (t^-)]\big),\quad
  \theta \in [0,1], \quad
  t\in \rmJ_q
\end{equation}
defines a connecting jump path
as desired in \eqref{e.BVc}.
We collect these remarks in the next proposition, whose proof
easily follows from Lemma \ref{le:chain} (see also
the Remark \ref{rem:nonde}).

\begin{proposition}
  \label{prop:equivalence}
   \RRSS In the setting of \eqref{assq},
\eqref{asserre}, \eqref{assene}, \RREE
  let $q_{\BV}\in \BV([0,T];\calQ)$ be a BV solution of the
  rate-independent system $(\calQ,d,\calE)$ and let $\wh \gph=(\wh t,\wh q)$
  be a map as in Lemma \ref{le:chain}.  Then, setting
  \begin{equation}
    \label{eq:8}
    \wh z(s):= \left\{\ba{cl} \wh q(s)&\text{if }s\in \wh\rmC_q\,,\\
      y^t(\theta)&\text{if }s\in \wh\rmJ_q,\
      \wh t(s)=t,\
      s=(1-\theta)\sigma (t^-)+\theta \sigma (t^+) \text{ for }
      \theta\in[0,1]\,,\ea \right.
  \end{equation}
  the map $(\wh t,\wh z):[0,S]\to \extQ$ is a
  parametrized metric solution of $(\calQ,d,\calE)$
  according to Definition \ref{def3.2}.

  Conversely, if $(\wh t,\wh q):[s_0,s_1]\to \extQ$ is a
   surjective  parametrized metric solution
   (i.e.\ $\wh t(s_0)=0$ and $\wh t(S_1)=t)$,
   then any map $q$ defined as in \eqref{eq:7}
  is a BV solution.
\end{proposition}

 The next result shows that  BV solutions can be directly obtained
as a vanishing viscosity limit, as
in Theorem \ref{3.3},  but now rescaling is not  needed.
The imposed \emph{a priori} bound on the total variation  for the
viscosity solutions
$q_\eps$ can be easily obtained from the energy inequality
\eqref{e3.9} under general assumptions on $\calE$, see e.g.\
\eqref{e:en-2}.

\GCOMMENT{\OLD{For this theorem we do not need the nondegeneracy
  condition;
  the convergence holds even if there would be a loss in
  the total length of the limit solution.
  Under the chain rule assumption, this loss is prevented,
  but in principle we could have some convergence result
  even if the chain rule does not hold, e.g. in
  for bad functionals where one should also take the
  l.s.c. relaxation of the slope.
  Of course, in this case one has a weaker
  notion of solution.}}

\begin{corollary}[Vanishing viscosity limit (II)]
  \label{cor:viscosity}
  Let $(\calQ,d,\calE)$ satisfy \eqref{assq}, \eqref{asserre}, and
  \eqref{assene}. For every $\eps >0$ let $q_\eps\in\AC([0,T]; \calQ)$ be a
  solution to \eqref{e3.9} for $\psi=\psi_\eps$. Assume that $q_\eps(0)\to
  q_0$ as $\eps\searrow0$ and $\VAR(q_\eps,[0,T])\le C$ for all $\eps>0$ with
  a constant $C$ independent of $\eps$.  Then, there exist a subsequence
  $q_{\eps_k}$ with $\eps_k\searrow0$ and a BV solution $q$ for
  $(\calQ,d,\calE)$ such that $q_{\eps_k}(t)\to q(t)$ as $k\to\infty$ for
  all $t\in [0,T]$.
\end{corollary}
\begin{proof}
  Let us consider the functions $\sigma_\eps$ as in Lemma
  \ref{le:chain}. By Helly's selection theorem we
  can find subsequences $(q_{\eps_k})_k,(\sigma_{\eps_k})_k$
  converging pointwise in $[0,T]$.
  Let us consider the corresponding parametrized metric solutions
  $(\wh t_\eps,\wh q_\eps)$
  introduced in  Proposition \ref{prop:equivalence}.
  Since $\sigma_\eps$ is absolutely continuous with $\sigma_\eps'\ge1$,
  differentiating the identity
  $\sigma_\eps(t)=\wh t_\eps(\sigma_\eps(t))+
  \wh\rho_\eps(\sigma_\eps(t))$ we obtain
  $m_\eps:=\wh t_\eps'+|\wh q_\eps'|=1$ a.e.\ in $(0,S_\eps)$ and
  $S_\eps:=\sigma_\eps(T)\ge T$.
  Since $S_{\eps_k}$ converges to $S\ge T>0$, up to a further linear rescaling
  it is not restrictive to assume that $S_{\eps_k}=S$ and $m_{\eps_k}=
  S_{\eps_k}/S\to 1$.

  Applying Theorem \ref{3.3} we can find suitable subsequences
  (still labelled $\eps_k$) such that
  $( \wh{t}_{\eps_k}, \wh{q}_{\eps_k})  \to (\wh{t} , \wh q)\
  \text{ in } \rmC^0([0,S];\extQ))$.
  Since $q_\eps(t)=\wh q_\eps(\sigma_\eps (t))$ and
  $t=\wh t_\eps(\sigma_\eps(t))$, we easily get
  $q_{\eps_k}(t)\to \wh q(\sigma(t))$ and $\wh t(\sigma(t))=t$,
  so that $q$ is a BV solution induced by $(\wh t,\wh q)$ as in
  \eqref{eq:7}.
  \end{proof}

\OLD{Using this identification between parametrized metric and BV solutions
we also see easily that the construction of $\Sigma$ was such that along the
corresponding solutions we have
\begin{equation}\label{e.SigmaPM_BV}
\Sigma(q,[\wh t(\sigma_1),\wh t(\sigma_1)]) = \int_{\sigma_1}^{\sigma_1}
\Slope\calE{\wh t(\sigma)}{\wh q(\sigma)}\,|\wh q'(\sigma)| \dd \sigma.
\end{equation}}

%\begin{remark}
%\label{rem:justif-chain-rule} \upshape Exploiting the above
%transformations to pass from an admissible parametrized curve
%$(\wh{t},\wh{q})$ to a BV curve $q$, it can be checked that, if
%the chain rule inequality \eqref{e3.6} holds, then $\calE$ also
%complies with \eqref{e4.5}.
%\end{remark}
%\COMMENT{Would it be possible to give a more ``intrinsic'' proof of
%chain rule \eqref{e4.5},  in order to help intuition? For example,
%in the simple case  in which $\calQ$ is  a finite-dimensional
%manifold ($\calQ=\R^N$?) and $\calE \in {\rmC}^1 (\extQ)$.
%Then, given a curve}

We conclude the section with the \textbf{Proof of Lemma \ref{le:chain}:} \\
Part \textbf{A)} is immediate. Part \textbf{B)} is an obvious extension of
\cite[2.5.16]{Fede69GMT}, since each couple of points in $\calQ$ can be
connected by a geodesic.  Notice that $\wh\tau(V_q(t))=t$ and therefore
$\wh\tau(\wh\rho (\sigma (t)))=t$ if $t\in \rmC_q$. We thus get $\wh\tau\circ
\wh\rho = \wh t$ in $\wh\rmC_q$.

  In order to prove \textbf{C)} it is not restrictive to assume
  $A=[0,T]$ and $\zeta\in \rmC^0_{\rmc}([0,T])$.
  Then, \eqref{e.change_of_var} follows from
  \eqref{e.Stieltjes} and \cite[2.5.18(3)]{Fede69GMT},
  observing that
  \begin{displaymath}
    \calL\big(\set{s\in [0,S]}{\wh t(s)<t}\big)=V_q(t)
    \text{ for all }t\in \rmC_q,\quad
    \calL\big(\set{r\in [0,R]}{\wh\tau(r)<t}\big)=V_q(t).
  \end{displaymath}
  Let us now prove the first identity of \eqref{e.densities}:
  setting $\wh\rmJ_q:= \wh t^{-1}(\rmJ_q)=[0,S]\setminus \wh\rmC_q$,
  we observe that
  $\wh t'(s)=0$ for $\calL$-a.e.\ $s\in \wh\rmJ_q$.
  Since $\wh t$ is Lipschitz continuous and monotone,
  the change of variable formula and \eqref{e.change_of_var} yield, for every
  continuous function $\zeta $ with compact support in $(0,T)$,
  \begin{displaymath}
    \int_0^T \zeta(t)\dd t=\int_0^S \zeta(\wh t(s))\wh t'(s)\dd s=
    \int_{\wh\rmC_q} \zeta(\wh t(s))\wh t'(s)\dd s=
    \int_{\rm C_q} \zeta(t)\, \wh t'\circ \sigma (t)\,{\RRSS \mu_{\gph}^\cont(\rmd t)
    . \RREE}
  \end{displaymath}
  The second identity of \eqref{e.densities} follows by a similar
  argument:
  \begin{align*}
     \int_0^T \zeta(t)\, \mu^\cont_q(\rmd t)&=
     \int_{\rmC_q} \zeta(t)\, \mu_q(\rmd t)=
    \int_{\wh\tau^{-1}(\rmC_q)} \zeta(\wh\tau(r))\dd r=
    \int_{\wh\rho^{-1}(\wh\tau^{-1}(\rmC_q))} \zeta(\wh\tau(\wh\rho (s)))
    \wh\rho'(s) \dd s\\&=
    \int_{\wh \rmC_q}\zeta(\wh t(s))\wh\rho'(s)\dd s=
    \int_{\rmC_q} \zeta(t) \wh\rho'\circ \sigma (t)\,\mu_x(\rmd t)=
    \int_0^T \zeta(t) \wh\rho'\circ \sigma (t)\,{\RRSS \mu_{\gph}^\cont(\rmd t)
    . \RREE}
  \end{align*}
  The identity $\wh\rho'\circ \sigma =|\wh q'|\circ \sigma $
  follows from the property $V_q(t)=V_{\wh q}(\sigma (t))$ for all $t\in
  \rmC_q$,
  so that $V_{\wh q}(s)=\wh\rho (s)$ for all $s\in \wh\rmC_q$. \hfill
  \rule{1ex}{1ex}

\section{Other solution concepts}
\label{s:other.sol}

%\subsection{Different solution types without parametrization}
%\label{ss:no.para}

Here we discuss other notions of solutions for rate-independent
systems $(\calQ,d,\calE)$, namely energetic solutions, local and
approximable solutions, and $\Phi$-minimal solutions.

\subsection{Energetic solutions}
\label{ss:en-sln}

The concept of energetic solutions provides the most general
setting, in the sense that it does not even rely on a
differentiability structure like the Finsler metric $\calR$, but
only uses
 the distance $d$. \RRSS In such a framework \RREE it is even
possible to consider quasi-metrics (i.e. unsymmetric and allowed to
take the value $\infty$), cf.\ \cite{Miel05ERIS}.

\begin{definition}\label{d:4.1}
  A mapping $q:[0,T]\to \calQ$ is called \emph{energetic solution} for
  the rate-independent system $(\calQ,d,\calE)$ if for all $t\in[0,T]$
  the global stability \ITEM{(S)} and the energy balance \ITEM{(E)} hold:
\[
\ba{l}
\ITEM{(S)}\quad\forall\, \wt{q}\in
\calQ:\quad\quad\calE (t,q(t))\leq
\calE (t,\wt{q})+d(q(t),\wt{q});\\
\ITEM{(E)}\quad\calE (t,q(t))+\VAR(q,[0,t])
=\calE (0,q(0))+\int^t_0\pl_s\calE (s,q(s))\dd s\,.
\ea
\]
\end{definition}
We refer to \cite{MieThe99MMRI,MiThLe02VFRI} for the origins of this
theory and to \cite{Miel05ERIS} for a survey. In analogy with
\eqref{e.BVjump} we have the jump relations
\begin{equation}\label{e.SEjump}
\ba{l}
\calE(t,q(t^-)){-}\calE(t,q(t))=d(q(t^-),q(t)), \\
\calE(t,q(t)){-}\calE(t,q(t^+))=d(q(t),q(t^+)), \\
\calE(t,q(t^-)){-}\calE(t,q(t^+))=d(q(t^-),q(t^+)),
\ea
\end{equation}
for all $t\in \rmJ_q$. Here they are easily obtained by considering
the energy identity $\calE (s,q(s))+\VAR(q,[r,s]) =\calE
(r,q(r))+\int^s_r\pl_\tau\calE (\tau,q(\tau))\dd \tau$, which
follows immediately \RRSS from \RREE (E), for the intervals
$[t{-}h,t]$, $[t,t{+}h]$, and $[t{-}h,t{+}h]$, respectively, and
letting $h\searrow 0$.

To compare energetic and BV solutions, we
introduce the \emph{global slope} $\LIP[\calE(t,\cdot)]:\calQ\to
[0,\infty]$ via
\[
\LIP[\calE(t,\cdot)] (q):= \sup_{\wt{q} \neq q} \frac{\big(\calE
(t,q)-\calE (t,\wt{q})\big)^+}{d(q,\wt{q})} \quad {\RRSS\text{for
all$\, q \in D\,.$}\RREE}
\]
Using this definition, the global stability (S) \RRSS can obviously
be rephrased \RREE as $\LIP[\calE(t,\cdot)] (q(t)) \leq 1$. We also
have
\begin{equation}
\label{glob-loc}
\Slope\calE tq \leq \LIP[\calE(t,\cdot)] (q)
\quad \text{ for all } (t,q)\in \extQ \,.
\end{equation}
Indeed, choosing a local coordinate system  in
$\calQ $, one can check (cf.\ \cite[Ch.VI.2]{BCS00IRFG}) that
\begin{equation}\label{e3.8}
\Slope\calE tq=\sup_{v\in
  \rmT_q\calQ \setminus\{0\}}
\frac{\langle\rmD_q\calE (t,q),v\rangle}{\calR_1(q,v)}
  =\limsup_{\wt{q}\to q}\frac{(\calE (t,q) -
\calE (t,\wt{q}))^+}{d(q,\wt{q})} \,,
 \end{equation}
whence \eqref{glob-loc}.

\begin{remark}\label{rm:EnergSln}
  \upshape
It is well known that (S) implies the lower energy estimate
\[
\calE(s,q(s))-\calE(r,q(r))-\int_r^s \pl_t\calE(t,q(t))\dd t \geq -
\VAR(q,[r,s])
\]
for $0\leq r<s\leq T$, cf.\ \cite[Thm.\,2.5]{MiThLe02VFRI} and
\cite[Prop.\,5.7]{Miel05ERIS}. In the present setting this is in
fact an easy consequence of the chain rule inequality \eqref{e4.5}
and of the observation that $\LIP[\calE(t,\cdot)](q(t))\leq C$
implies $\Gamma(q,[r,s]) \leq C\VAR(q,[r,s])$.

Moreover, it is possible to derive a gradient-flow like inequality
of the type given in \eqref{e3.9},  \eqref{e.DefPaMeSln_c}, or
\eqref{e.BVineq}. For this, define the functional
$\Gamma_*(\cdot,[r,s])$ on $\BV([0,T],\calQ)$ via
\begin{align*}
\Gamma_*(q,[r,s]):={}&
\int_r^s \max\big\{ \LIP[\calE(t,\cdot)](q(t)),1\big\} \mu^\cont_q(\dd
t)
+
\int_r^s\big( \LIP[\calE(t,\cdot)](q(t))-1\big)^+\dd t
\\
 &+ \max\{|\calE(r,q(r)){-}\calE(r,q(r^+))|,d(q(r),q(r^+))\} \\
 & +  \max\{|\calE(s,q(s^-)){-}\calE(s,q(s))|,d(q(s^-),q(s))\} \\
 &+\sum_{t\in(r,s)\cap \rmJ_q} \big[\max\{|\calE(t,q(t^-)){-}\calE(t,q(t))|,
   d(q(t^-),q(t))\}\\[-0.5em]
 &\hspace{4em} +
     \max\{|\calE(t,q(t)){-}\calE(t,q(t^+))|,d(q(t),q(t^+))\} \big]\,.
\end{align*}
Then, $q:[0,T] \to \calQ$ is an energetic solution for $(\calQ,d,\calE)$ if
and only if $\LIP[\calE(0,\cdot)](q(0))\leq 1$
and
\begin{equation}\label{e.ESineq}
\calE (s,q(s))-\calE (r,q(r))-\int^s_r\pl_t\calE (t,q(t))\dd t
 \leq -\Gamma_*(q,[r,s]) \text{ for } 0\leq r < s\leq T\,.
\end{equation}
\end{remark}

The following result essentially states that every energetic
solution $q$ is  a BV solution outside its jump set. Moreover, if
the jump relations \eqref{e.BVjump} and \eqref{e.SEjump} are both
satisfied, then an energetic solution is also a BV solution.
Conversely, if a BV solution additionally satisfies
$\LIP[\calE(t,\cdot)](q(t))\leq 1 $, then it is  an energetic
solution as well.

\begin{proposition}[Comparison between energetic and BV solutions]
\label{p:ene-bv}
Assume that $(\calQ,d,\calE)$ satisfies \eqref{assq}, \eqref{asserre},
and \eqref{assene}.
\\[0.3em]
(A) If $q \in \BV([0,T];\calQ)$ is an energetic solution, then $q$ is also a BV
solution if and only if \eqref{e.BVc} holds additionally.
\\[0.3em]
(B) If $q \in \BV([0,T];\calQ)$ is a BV solution with
$\LIP[\calE(t,\cdot)](q(t)) \leq \max\{1,\Slope\calE t{q(t)}\} $
for all $ t \in [0,T]$ and $\LIP[\calE(0,\cdot)](q(0))\leq 1$,
then $q$ is also an energetic solution.

% \COMMENT{The following OLD statement ``1.'' is obviously wrong:

% {\upshape 1. If $q \in \BV([0,T];\calQ)$ is an energetic solution, then
% \begin{subequations}
% \label{en-BV}
% \begin{equation}
% \label{en-BV-1} \Slope\calE t{q(t)} \leq 1 \quad
% \forall\, t \in [0,T]\,,
% \end{equation}
% \begin{equation}
% \label{en-BV-2}
%  \LIP[\calE(t,\cdot)]
% (q(t))=\Slope\calE t{q(t)}=1 \qquad \forall t \in
%  \sppt (\mu_q)\,,
% \end{equation}\end{subequations}
% and $q$ fulfils \eqref{e.BVen} with an
% equality sign.}

% The counterexample is as follows: Take $\calQ=\R$, $d$ the usual distance, and
% $\calE(t,q)=\frac12((t{-}1)^2-3)q^2+\frac12q^4$. Then $q:[0,2]\to \R$ with
% $q(t)=0$ for $t\leq 1$ and $q(t)=1$ for $t>1$ is an energetic solution.
% However, the slope is $\Slope\calE t{q(t)}=0$ for $t\leq 1$ and $(t{-}1)^2-1$
% for $t>1$. On $\sppt \mu_q=\{1\}$  the `local slope' equals
% $0<1$.

% In fact, for this example \eqref{e.BVen} is satisfied with equality, but we
% may change $\calE$ slightly on the $q$-interval $[1/3,2/3]$ by adding a bump
% such that $\int_0^1 \Slope\calE ty \dd y > \calE(1,0)-\calE(1,1)$, then
% \eqref{e.BVen} is no longer valid but the energetic solution remains the
% same.
% }
\end{proposition}
\begin{proof}
To prove (A), we first note that the necessity of \eqref{e.BVc} is
obvious at it is part of the definition of BV solutions. To
establish the sufficiency in (A), we observe that \eqref{e.BVc}
yields \eqref{e.BVjump}, so that the dissipation term $\Sigma_1$ of
Proposition \ref{prop:BV_alter} satisfies
\begin{displaymath}
  \Sigma_1(q,[t_0,t_1])\le \Gamma_*(q,[t_0,t_1])\quad
  \text{for all }0\le t_0<t_1\le T.
\end{displaymath}
\RRSS Hence, \RREE \eqref{e.BVineq} follows from \eqref{e.ESineq}.
Thus, statement (A) is established.

The necessity of the additional condition on $\LIP[\calE(t,\cdot)](q(t))$ is
obvious, since energetic solutions have to satisfy the stronger stability
condition (S). To show the sufficiency
we observe that the additional condition
yields
\begin{displaymath}
  \Gamma_*(q,[t_0,t_1])\le \Sigma_1(q,[t_0,t_1])\quad
  \text{for all }0\le t_0<t_1\le T,
\end{displaymath}
so that \eqref{e.ESineq} follows from \eqref{e.BVineq}.
\end{proof}

\begin{remark}
\label{r:geod-conv} \upshape
The additional condition in Proposition \ref{p:ene-bv}(B) is implied
by the general condition
\begin{equation}\label{e:Lip-slope}
\LIP[\calE(t,\cdot)](\wh q)=
\Slope\calE t{\wh q} \quad
\text{ for all } (t,\wh q)\in \extQ\,.
\end{equation}
If this condition holds, then the notions of energetic solutions and
BV solutions coincide under the additional assumption that the
initial state $q_0$ is stable, i.e.\ $\LIP[\calE(0,\cdot)](q_0)\leq
1$.  One condition guaranteeing \eqref{e:Lip-slope} is  a
metric version of convexity  for $\calE(t,\cdot)$, see
\cite[Def.\,2.4.3]{AmGiSa05GFMS}.  Here, we say that ${\RRSS
\mathcal{F} \RREE}:\calQ\to \R\cup\{\infty\}$ is convex on
$(\calQ,d)$, if
\begin{equation}
\begin{aligned}
 \forall\, q_0,q_1\in \calQ,\ \theta\in [0,1]\
 \exists\,q_\theta\in\calQ:
 \ &
d(q_0,q_\theta)=\theta d(q_0,q_1),\ d(q_\theta,q_1)=(1{-}\theta
)d(q_0,q_1), \\ & {\RRSS \mathcal{F} \RREE}(q_\theta)\leq
(1{-}\theta) \, {\RRSS \mathcal{F} \RREE}(q_0) + \theta \, {\RRSS
\mathcal{F} \RREE}(q_1) .
\end{aligned}\label{eq:13}
\end{equation}

\GCOMMENT{\OLD{The term ``geodesically'' is not completely correct here,
since $q_\theta$ is not necessarily a geodesic. On the other hand,
the old definition impose convexity on all geodesics and implicitely
assume the existence of at least one geodesic.}}

To establish \eqref{e:Lip-slope} for ${\RRSS \mathcal{F} \RREE}$
note that for each $\wh q$ and $\eps>0$ we have $\LIP[{\RRSS
\mathcal{F} \RREE}](\wh q)\geq \frac{{\RRSS \mathcal{F} \RREE}(\wh
q){-}{\RRSS \mathcal{F} \RREE}(\wt q)}{d(\wh q,\wt q)}{-}\eps$ for
some $\wt q$. Moreover, for each $n\in \N$, choosing $\theta=1/n$ we
find $q_n$ with $d(\wh q,q_n)=\frac1n d(\wh q,\wt q)$ and $d(q_n,\wt
q)=\frac{n-1}{n} d(\wh q,\wt q)$. Applying \eqref{eq:13} we obtain
\begin{align*}
&|\pl_q{\RRSS \mathcal{F} \RREE}|(\wh q)\geq \limsup_{n\to \infty}
              \frac{{\RRSS
\mathcal{F} \RREE}(\wh q){-}{\RRSS \mathcal{F} \RREE}(q_n)}{d(\wh
q,q_n)}
 \geq \limsup_{n\to \infty}
   \frac{{\RRSS
\mathcal{F} \RREE}(\wh q){-} \frac{n-1}{n}{\RRSS \mathcal{F}
\RREE}(\wh q)-
    \frac1n {\RRSS
\mathcal{F} \RREE}(\wt q)}{\frac1n\,d(\wh q,\wt q)}\\
&= \frac{{\RRSS \mathcal{F} \RREE}(\wh q){-}{\RRSS \mathcal{F}
\RREE}(\wt q)}{d(\wh q,\wt q)} \geq \LIP[{\RRSS \mathcal{F}
\RREE}](\wh q){-}\eps \,.
\end{align*}
 In this way, part (B) of Proposition~\ref{p:ene-bv} is a generalization
to the metric setting of Theorem 3.5 in \cite{MieThe04RIHM}, which states
that for a Banach space $\calQ$, a convex energy functional $\calE$, and a
translation invariant metric $d$ the subdifferential formulation and the
energetic formulation for $(\calQ,d,\calE)$ are equivalent.
\end{remark}

\subsection{Local  and approximable solutions}
\label{ss:local-approx}

As we  have already mentioned in the introduction, energetic
solutions have the disadvantage that the stability condition {(S)}
is global, so that  solutions tend to jump earlier than expected,
see Example~\ref{ex:5.1}. To avoid these early jumps, the vanishing
viscosity method was employed in
\cite{EfeMie06RILS,DDMM07?VVAQ,ToaZan06?AVAQ,KnMiZa07?ILMC}. When
avoiding parametrization and studying the limits of the viscous
solutions $q_\eps:[0,T]\to \calQ$ directly, one obtains an {\em
energy inequality} and  a \emph{local stability} condition.  Hence,
we next introduce the notions of \emph{local} solution and of {\em
approximable} solution, generalizing the definitions given in
\cite{ToaZan06?AVAQ} to the metric setting.
%%%%%%%%%%%%%%%

\begin{definition}\label{d4.2}
 A mapping $q:[0,T]\to \calQ$ is called \emph{local solution}, if
\ITEM{(a1)} and \ITEM{(a2)} hold:
\[
\ba{ll}
\ITEM{(a1)}& \Slope\calE t{q(t)} \leq 1
 \quad\foraa t\in[0,T];\\
\ITEM{(a2)}& \text{for all }  r,s\in[0,T]\text{ with }
   r<s\text{ we have}\\
  & \calE (s,q(s))+\VAR(q,[r,s]) \leq
  \calE(r,q(r))+\int^s_r\pl_t\calE (\tau,q(\tau))\dd \tau.
\ea
\]
\end{definition}
\noindent We will see in the examples of~Section~\ref{s:examples}
that the notion of local solution is very general.  Using
\eqref{glob-loc}, it is clear that all energetic solutions are local
solutions. Similarly, all BV solutions are local solutions. \RRSS To
see  this, we  use \RREE \eqref{e.BVb} and \eqref{e.BVc} to obtain
\ITEM{(a1)}, since there are at most a countable number of jump
points, and we conclude ${\RRSS \Sigma_0 \RREE}(q,[r,s]) \geq
\VAR(q,[r,s])$ for $ 0 \leq r\leq s \leq T$, which gives
\ITEM{(a2)}.

On the other hand, note that, unlike the case of energetic
solutions, the combination of the local stability condition with the
energy inequality does not provide full information on the solution
$q$. In particular, the behavior of the solution at jumps is poorly
described by relations \ITEM{(a1)} and \ITEM{(a2)}.  This also
highlights the role of the term ${\RRSS \Sigma_0 \RREE} (q,\cdot)$,
here missing, in the energy identity for BV solutions. As a
consequence there are many more local solutions, see also Example
\ref{ex:5.1}.

Indeed, the vanishing viscosity method turns out to provide a
selection criterion for local solutions. Among local solutions, we
thus distinguish the following ones:

\begin{definition}
\label{d4.5} A mapping $q:[0,T]\to \calQ$ is called \emph{approximable
solution}, if there exists a sequence $(\eps_k)_{k\in\N}$ with
$\eps_k\searrow 0$ and solutions $q_{\eps_k}\in\AC([0,T],\calQ)$
of \eqref{e3.9} with $\psi=\psi_{\eps_k}$ such that for all
$t\in[0,T]$ we have $q_{\eps_k}(t)\to q(t)$.
\end{definition}

%\begin{remark}[Approximable and BV solutions]\label{rem:appr-bv} \upshape
%
It follows from Corollary~\ref{cor:viscosity} that,
under the assumptions
\eqref{assq}, \eqref{asserre}, and \eqref{assene},
% we may associate with
% every approximable solution $q$ on some interval $(s_0,s_1)$ a suitable
% parametrization $t$ in such a way that $(t,q)$ is a parametrized metric
% solution of the rate-independent system $(\calQ,d,\calE)$.  Therefore, in
% view of the links between parametrized and BV solutions discussed in
% Section~\ref{ss:BB.sol}, we conclude that
any approximable solution is a BV solution as well.
%
%\end{remark}

The notion of approximable solutions suffers from two drawbacks. First of all,
there is no direct characterization of the limits in terms of a
subdifferential inclusion or variational inequality, unlike for
parametrized/BV solutions, recall Proposition~\ref{p3.equiv.evol}.
Secondly, since the solution set is defined through a limit procedure, it is
not upper semicontinuous with respect to small perturbations, as shown in
Example~\ref{ex:5.3}. This is in contrast with the stability properties of the
set of parametrized/BV solutions, see
Remark~\ref{rem:point-forward-stability}.

\subsection{Visintin's $\Phi$-minimal solutions}
\label{ss:Vis.sol}

In \cite{Visi01NAE,Visi06?MPE} a new minimality principle was
introduced. Here, we present the adaptation to  rate-independent
evolutions proposed in \cite[Sect.5.4]{Miel05ERIS}, in the current
metric setting. Again, we use parametrized curves, as it is
essential to have continuous paths. For simplicity, we restrict to
arclength parametrization, i.e., \begin{equation}\label{e4.10}
t'(s)+|q'|(s)=1\quad\foraa  s\in [s_0,s_1]. \end{equation}  In the
framework of \eqref{assq}, \eqref{asserre}, \eqref{assene}, for a
given initial pair $(t_0,q_0)$ we introduce the space of
arclength-parametrized paths (on some interval $[0,S]$) starting in
$(t_0,q_0)$ via
\[
\ba{ll}
\calA_S(t_0,q_0):=&\set{(t,q)\in\rmC^0([0,S],\extQ)}{t(0)=t_0,
 \ t \text{ nondecreasing}, \\
&\quad q(0)=q_0, \ t(s)+\VAR(q,[0,s])=s\quad\text{for all
}s\in[0,S]}. \ea
\]
On this set we define the function
$\Phi:\calA_S(t_0,q_0)\to\rmL^\infty([0,S])$ via
\[
\Phi[t,q](s)=\calE (t(s),q(s))+\VAR(q,[0,s])
-\int_0^s\pl_t\calE (t(r),q(r))t'(r)\dd r \,.
\]
Between paths in ${\RRSS \calA_S \RREE}(t_0,q_0)$ we \RRSS
introduce \RREE an order relation $\preceq$ as follows. For $(t,q)$,
$(\tau,p)\in\calA_S(t_0,q_0)$, define the ``arclength of equality''
via
\[
\bfS[(t,q),(\tau,p)]=\inf\set{s\in[0,S]}{(t(s),q(s))\not=(\tau(s),p(s))}.
\]
 Then, the
order relation is given by
\[
(t,q)\preceq(\tau,p) \quad \Longleftrightarrow  \quad
\left\{
\ba{l}
\forall\, s> \bfS[(t,q),(\tau,p)]\  \ \exists\,
s_*\in (\bfS[(t,q),(\tau,p)],s):\\
\quad\; \; \Phi[(t,q)](s_*)\leq\Phi[(\tau,p)](s_*).
\ea \right.
\]

\begin{definition}\label{d:4.5}
  An arclength-parametrized function $(t,q):\, [s_0,s_1]\to\extQ$ is
  called a \emph{$\Phi$-minimal solution} for $(\calQ,d,\calE )$, if for all
  $(\tau,p)\in\calA_{s_1-s_0}(t(s_0),q(s_0))$ we have $(t,q)\preceq(\tau,p)$.
\end{definition}

Like  for  energetic solutions, this solution notion appears
particularly suitable to handle nonsmooth energy functionals, since
no derivatives/slopes  of $\calE$ with respect to the variable $q$
occur in the definition of the functional $\Phi$.

We  now show that, in a smooth setting,
$\Phi$-minimal solutions are parametrized metric solutions.
%
%\paragraph{$\Phi$-minimal and  parametrized metric solutions.}
%To fix ideas, assume that $(\calQ,d,\calE)$ are as in
%Section~\ref{s:mech.mot} (see
%also \eqref{e3.7}--\eqref{classical-ch-rule}).
 Using suitably chosen test functions, it can be shown that
a necessary condition for $\Phi$-minimality is the local condition
\[
\frac{\rmd}{\rmd s}\Phi[(t,q)](s)\leq
 \calN (t(s),q(s))\quad\foraa s \in (s_0,s_1),
\]
where $ \calN :\extQ\to \R$ is defined via
\[
 \calN (t,q)=\liminf_{\eps\to 0} \Big( \frac1\eps
\inf\set{\calE (t,\wt{q}){-}\calE
  (t,q){+}d(q,\wt{q})}{d(q,\wt{q})\leq\eps}\Big)\,.
\]
A simple calculation gives
\[
 \calN (t,q)=\left\{ \ba{cl}
0&\text{for }\Slope\calE tq\leq 1,\\
1-\Slope\calE tq&\text{for
}\Slope\calE tq\geq 1\ea \right. \quad \forall\,
(t,q) \in \extQ\,.
\]
Since  $ t'+|q'|=1$ a.e.\ and
\[
\frac{\rmd}{\rmd s}\Phi[(t,q)](s)=\frac{\rmd}{\rmd
s}\calE (t(s),q(s))+|q'|(s)-\pl_t\calE (t(s),q(s))t'(s)
\quad\foraa s \in (s_0,s_1),
\]
we conclude that all $\Phi$-minimal solutions satisfy a.e.\ in
$(s_0,s_1)$
\[
\frac{\rmd}{\rmd
s}\calE (t(s),q(s))-\pl_t\calE (t(s),q(s))t'(s)\leq
-\wt{M}\big(t'(s),|q'|(s),\Slope\calE {t(s)}{q(s)}\big),
\]
together with the constraint $t'(s)+|q'|(s)=1$, where
$\wt{M}(\alpha,\nu,\xi)=\nu+(\xi-1)^+$. We have thus proved that any
$\Phi$-minimal solution $(t,q)$ on $[s_0,s_1]$ is  a parametrized
metric solution, and hence a BV solution (up to a
parametrization).
 The
opposite is in general not true, see Example~\ref{ex:5.2}. Further,
Example~\ref{ex:5.3} shows that the set of $\Phi$-minimal solutions
is not stable with respect to perturbations.

\section{Outlook to the analysis in metric spaces}
\label{s:outlook}

In \cite{MiRoSa08?VVLM} we shall analyze rate-independent evolutions in
\begin{equation}
\label{ass-metric} \text{a complete metric space $(\mespace,d)$\,.}
\end{equation}
In fact, using the results in \cite{RoMiSa08MACD} we shall be able to
handle the case in which $d$ is a quasi-distance on $\mespace$, i.e.\
possibly unsymmetric and possibly taking the value $\infty$.

In this framework, the \emph{metric velocity} \eqref{e:1-mder} of
a curve $q\in\AC([0,T];\mespace)$ is defined by
\begin{equation}\label{e3.7}
|q'|(t):=\lim_{h\searrow 0}\frac{1}{h}d(q(t),q(t{+}h))=
\lim_{h\searrow 0}\frac{1}{h}d(q(t{-}h),q(t)) \quad \foraa\, t \in
(0,T)\,.
\end{equation}
In the Finsler setting \eqref{assq}--\eqref{asserre} of Section~\ref{sez3.1},
one indeed has $|q'|(t)=\calR_1 (q(t),q'(t))$ for a.a.\ $t \in
(0,T)$ (see \cite[Chap.VI.2]{BCS00IRFG}). Further, given a
functional
 \begin{equation}
\label{e:en-1}
 \calE  :\ [0,T]\ti \mespace\to
\R\cup\{\infty\}, \  \text{ with domain }    \mathrm{dom}(\calE)=
[0,T] \times D
\end{equation}
 for some $D \subset \mespace$, the \emph{local slope} is defined via
\begin{equation}
\label{e3.8-bis} \Slope\calE tq
  =\limsup_{\wt{q}\to q}\frac{1}{d(q,\wt{q})}(\calE (t,q) -
\calE (t,\wt{q}))^+\,,
\end{equation}
which in the Finsler setting $\mespace=\calQ$ coincides with
\eqref{newform2}. With these tools, \eqref{e3.9} is the purely
metric formulation of doubly nonlinear equations of the type
\eqref{dne}, provided the local slope fulfils \RRSS the following
\RREE \emph{chain rule inequality}.

\begin{definition}
\label{d:3.1}  We say that the triple $(\mespace,d,\calE )$
satisfies the chain rule inequality if for every absolutely
continuous curve $(t,q): [s_0,s_1]\to[0,T]\ti \mespace$ such that
$t'\geq 0$ a.e. in $(s_0,s_1)$, $q(s)\in D$ for every $s \in
[s_0,s_1]$, and
\[
\int_{s_0}^{s_1}\Slope\calE {t(s)}{q(s)}\,|q'|(s)+
 |\pl_t\calE (t(s),q(s))|\,t'(s) \dd s<\infty\,,
\]
the map $s\mapsto\calE (t(s),q(s))$ \RRSS is absolutely continuous
on $[s_0,s_1]$ \RREE and satisfies
\begin{equation}\label{e3.6-bis}
\frac{\rmd}{\rmd s}\calE (t(s),q(s))-\pl_t\calE (t(s),q(s))t'(s) \geq -
\Slope\calE {t(s)}{q(s)}|q'|(s) \aein (s_0,s_1).
\end{equation}
\end{definition}

Unlike in Section~\ref{sez3.1}, within this abstract setting the
chain rule inequality is no longer granted, but has to be imposed
instead. We refer to \cite{RoMiSa08MACD} for a discussion on some
sufficient conditions for \eqref{e3.6-bis} to hold.

If the above chain rule holds, the (parametrized) metric formulation
\eqref{e3.11}  is again the starting point for the vanishing
viscosity analysis, which was developed in Section~\ref{ss3.2} in a
Finsler setting for smooth $\calE$. In this general framework,
$\calE$ has to satisfy some coercivity and (lower semi-) continuity
properties:
\begin{equation}
\label{e:en-2}
\begin{aligned}
&\text{the functionals $\calE(t,\cdot)$ are lower semicontinuous and}
\\
&\text{uniformly bounded from below with }
   K_0:=\inf_{t\in [0,T],\,q\in D} \calE(t,q)>-\infty\,;
\\[0.4em]
&\forall\, t \in [0,T]: \  \text{$\calE(t,\cdot)$ has compact sublevels
   in $\mespace$\,;}
\\[0.4em]
&\exists\, K_1 >0\;\forall \, q \in D:\ \calE(\cdot,q) \in
\rmC^1([0,T])
 \text{ and }
\\
&\hspace*{8.5em}  |\partial_t \calE(t,q)|\leq K_1
    (\calE(t,q){+}1) \text{ for all }t\in[0,T]\,;
\\[0.4em]
&\forall\, ((t_n,q_n))_{n\in \N} \subset [0,T]\ti \mespace \text{ with }
 (t_n,q_n) \to (t,q):
\\
& \pl_t \calE(t,q)=\lim_{n \to \infty}\pl_t \calE(t_n,q_n)  \quad
 \text{and}\quad
 \Slope\calE tq \leq \liminf_{n \to \infty}
 \Slope\calE{t_n}{q_n} \,.
\end{aligned}
\end{equation}
In \cite{MiRoSa08?VVLM}, under assumptions \eqref{ass-metric} and
\eqref{e:en-1}--\eqref{e:en-2} on $\calE$, we shall perform the
vanishing viscosity analysis of Theorem~\ref{3.3}, leading to the
notion of parametrized metric solution of the rate-independent
system $(\mespace,d, \calE)$. In this general setting, it is
obviously still possible to consider the notion of BV solution, and
our remarks on the comparison between BV (parametrized) and
local/approximable/$\Phi$-minimal solutions carry over.

Indeed, in \cite{MiRoSa08?VVLM} we shall    discuss BV solutions
with more detail, in particular proving existence  through
approximation by time discretization and solution of incremental
(local) minimization problems.

\section{Examples}
\label{s:examples}

Many of the differences between the various solution concepts
discussed above manifest themselves already in the case in which the
state space is the real line. Hence, we discuss the very simple
model with
\begin{equation}
\label{e:setting} \calQ=\R,\ \   d(q_0,q_1)=|q_0-q_1|\,, \ \  \calE (t,q)= U(q)-\ell(t)q\,,
\end{equation}
where the function $\ell$ will be  specified in the different
examples. The potential $U$ is the nonconvex function given via
\begin{equation}
\label{def:u} U(q)=\left\{ \ba{ll}
\frac{1}{2}(q{+}4)^2&\text{for } q\leq -2,\\
4{-}\frac{1}{2}q^2&\text{for } |q|\leq 2,\\
\frac{1}{2}(q{-}4)^2&\text{for } q\geq 2. \ea \right.
\end{equation}
As initial datum we shall take
\begin{equation}
\label{init-datum} q_0 = -5.
\end{equation}

\begin{example}\label{ex:5.1}\slshape
We let $\ell(t)=t$ for all $t \geq 0$. We claim that the
approximable, the $\Phi$-minimal, the parametrized, and the BV
solutions on $[0,\infty)$ are essentially unique and coincide.
However, the unique energetic solution is different. Moreover, we
show that there is an uncountable family of different local
solutions. With direct calculations, one sees that the energetic
solution takes the form
\[
q(t)=t{-}5\ \text{ for }t\in[0,1)\quad\text{and}\quad
q(t)=t{+}3\ \text{ for }t>1.
\]
Choose any $t_*\in[1,3]$ and any $q_*\in \big[ 3{+}t_*,
3{+}t_*{+}\min \{ 2,4\sqrt{t_*{-}1} \} \big]$. Then,
\[
q(t)=
\left\{
\ba{cl}
t{-5}&\text{for }t\in[0,t_*),\\
q_*&\text{for }t\in(t_*,q_*{-}3],\\
t{+}3&\text{for }t\geq q_*{-}3,
\ea \right.
\]
is a local solution. Note that the starting point of the jump at
$q(t_*-)=t_*{-}5$ can be chosen in a full interval. Moreover, for a
fixed $t_*>1$ we still have the possibility to choose the ending point
$q_*=q(t_*+)$ of the jump in a full interval.

All the other solution types essentially lead  (up to definition in
one point) to the same solution. Without time parametrization it
reads
\[
q(t)=
\left\{
\ba{cl}
t{-5}&\text{for }t\in[0,3),\\
q_*&\text{for }t=3,\\
t{+}3&\text{for }t> 3,
\ea \right.
\]
where $q_*\in[-2,6]$ is arbitrary. The associated
arclength-parametrized solution takes the form
\[
\big(\wh{t}(s),\wh{q}(s)\big)=
\left\{
\ba{cl}
\big(\frac{s}{2},\frac{s}{2}{-}5\big)&\text{for }s\in[0,6],\\
(3,s{-}8)&\text{for }s\in[6,14],\\
\big(\frac{s}{2}{-}4,\frac{s}{2}{-}1\big)&\text{for }s\geq 14.
\ea \right.
\]
\end{example}

\begin{example}\label{ex:5.2}\slshape
  In this example we show that, in general, approximable solutions and
  $\Phi$-minimal solutions are different. In particular, recalling the
  discussions in Sections \ref{ss:local-approx} and \ref{ss:Vis.sol}, this
  shows that the set of BV solutions (or parametrized metric solutions) is
  strictly bigger then any of the other solution sets.

In the setting of \eqref{e:setting}--\eqref{init-datum},  we now
choose the function $\ell(t):=\min\{t,6{-}t\}$ for all $t \geq 0$,
i.e., the loading reduces exactly when the solution reaches the jump
point. It is easy to see that there are two different BV solutions:
$q_1$, which jumps at $t=3$,  and $q_2$, which does not jump. We
have
\[
\ba{ll}
q_1(t)=
\left\{
\ba{cl}
t{-}5&\text{for }\ t\in[0,3),\\
6&\text{for }\ t\in(3,5],\\
11{-}t&\text{for }\ t\in[5,9),\\
3{-}t&\text{for }t\geq 9; \ea \right. & q_2(t)= \left\{ \ba{cl}
t{-}5&\text{for } t\in[0,3],\\
-2&\text{for } t\in[3,5],\\
3{-}t&\text{for } t\geq 5. \ea \right. \ea
\]
%\COMMENT{in the July version there was $t-3$ in the last line of the
%definition of $q_2$, but this would correspond  a different
%definition (also taking into account the jump at $t=3$) of
%$(\wh{t}_2,\wh(q)_2)$,  so I suppose that $3-t$ is the correct
%definition..}
 For $\eps>0$ the viscous solution $q^\eps$ of the
differential inclusion
\[
0\in\mathrm{Sign}(\dot{q})+\eps\dot{q}+U'(q)-\ell(t),\quad q(0)=-5,
\]
is unique and can be calculated by matching solutions of linear
ODEs. We find
\[
q^\eps(t)=
\left\{
\ba{cl}
t{-}5{+}\eps(\ee^{-t/\eps}{-}1)&\text{for }\ t\in[0,3],\\
q_*^\eps&\text{for }\ t\in[3,t_*^\eps],\\
3{-}t{+}\eps(\ee^{-(t{-}t_*^\eps)/\eps}{-}1)&\text{for }\ t\geq t_*^\eps,
\ea \right.
\]
where $q_*^\eps=q^\eps(3-)\lesssim -2$ and
$t_*^\eps=3-q_*^\eps\gtrsim 5$. Thus, we have $q^\eps(t)\to q_2(t)$
for every $t \geq 0$ as $\eps \downarrow 0$, and $q_2$ turns out to
be approximable, whereas $q_1$ is not. As a general principle, one
may conjecture that viscosity slows down solutions, and thus
approximable solutions tend to avoid jumps if there is a choice.

For $\Phi$-minimal solutions this seems to be opposite. We claim
that $q_1$ is (up to  a reparametrization) $\Phi$-minimal but $q_2$
is not. For this, we use the arclength parametrizations
\[
(\wh{t}_1,\wh{q}_1)(s)= \left\{\ba{cl}
\big(\frac{s}{2}, \frac{s}{2}{-}5\big)&\text{for }\ s\in[0,6],\\
(3, s{-}8)& \text{for }\ s\in[6,14],\\
(s{-}11, 6)&\text{for }\ s\in[14,16]; \ea \right.
 \quad
(\wh{t}_2,\wh{q}_2)(s)= \left\{\ba{cl}
\big(\frac{s}{2}, \frac{s}{2}{-}5\big)&\text{for }\ s\in[0,6],\\
(s{-}3,-2)&\text{for }\ s\in[6,8],\\
\big(\frac{s}{2}{+}1, 2{-}\frac{s}{2}\big)&\text{for }s\geq 8.
\ea \right.
\]
The functionals $\varphi_j(s)=\Phi[(\wh{t}_j,\wh{q}_j)](s)$ for all
$s \geq 0$, $j=1,2$, can be calculated explicitly: indeed, one
checks that
\begin{equation}
\label{calculations} \varphi_1 (s) =
\left\{\ba{cl} \frac12 & \text{for } s \in[0,6],
\\
\frac12-\frac12(s{-}6)^2 & \text{for } s\in[6,10],
\ea\right.
 \qquad \text{while} \qquad  \varphi_2 (s)= \frac12 \ \text{for } s
\geq 0\,,
\end{equation}
which clearly shows that $(\wh{t}_2,\wh{q}_2)$ is not
$\Phi$-minimal for $s\in[0,7]$.

To prove $\Phi$-minimality of \RRSS $(\wh t_1,\wh q_1)$ \RREE we
point out that chain rule inequality \eqref{e3.6} gives
\begin{equation}
\label{e:argument}
\Phi(\tau(s),p(s))\geq  \frac12 + \VAR(p,[0,s]) - \int_0^s
\Slope\calE{\tau(\sigma)}{p(\sigma)}|p'|(\sigma)\, \dd
\sigma
\end{equation}
for all $ s \in [0,T]$ and all $(\tau,p)\in\calA_{T}(0,q_0)$.
Equality holds in \eqref{e:argument} if and only if $(\tau,p)$ is  a
parametrized metric solution $(\wh{t},\wh{q})$ on $[0,T]$. In that
case, in view of \RRSS \eqref{eq:equivalent1} \RREE one further has,
for all $s \in [0,T]$,
\[
\Phi(\wh{t}(s),\wh{q}(s))= \frac12 + \text{\rm Var}(\wh{q},[0,s]) -
\int_0^s {\RRSS M \RREE} (|\wh{t}'|(\sigma), |\wh{q}'|(\sigma),
\Slope\calE{\wh{t}(\sigma)}{\wh{q}(\sigma)}|\wh{q}'|(\sigma)\, \dd
\sigma \,.
\]
Therefore, in order to check that $(\wh{t}_1,\wh{q}_1)$ is $\Phi$-minimal, it
is sufficient to prove that $ (\wh{t}_1,\wh{q}_1) \preceq (\wh{t},\wh{q}) $
for all parametrized metric solutions $(\wh{t},\wh{q})$, and this, for all the
arclength parametrizations $(\wh{t},\wh{q})$ corresponding to the (not
jumping) BV solution $q_2$, follows from the previous discussion on
$(\wh{t}_2,\wh{q}_2)$.  Now, the above energy balance states a general fact
about parametrized metric solutions: $\Phi(\wh{t},\wh{q})$ is constant as long
as no jumps occur, i.e.  $\Slope\calE{\wh{t}}{\wh{q}}\leq 1$ holds. If
jumps with $\Slope\calE {\wh{t}}{\wh{q}}> 1$ occur, then $\Phi$ will
strictly decrease. Thus, if there is a choice between one solution with a
fast jump and another without jumps, then the solution without
jumps cannot be $\Phi$-minimal.
\end{example}

\begin{example}\label{ex:5.3}\slshape
Here, we study the parameter dependence of solutions under the
loading
\[
\ell_\delta(t)=\min\{t, 6{+} 2\delta {-} t\} \quad \text{for } t \geq
0,
\]
where $\delta$ is a small parameter. In the case $\delta=0$ we have
two BV solutions $q_1$ and $q_2$ (or similarly parametrized metric
solutions), as was discussed in Example \ref{ex:5.2}. For
$-1<\delta<0$ there is only one solution, namely
\[
q^\delta(t)=
\left\{\ba{cl}
t{-}5&\text{for } t\in[0, 3{+}\delta],\\
\delta{-}2&\text{for } t\in[3{+}\delta, 5{+}\delta],\\
3{+}2\delta{-}t&\text{for } t\geq 5{+}\delta .
\ea \right.
\]
The corresponding parametrized solution is the unique $\Phi$-minimal
solution. Now, for $\delta\nearrow 0$ we find $q^\delta(t)\to
q_2(t)$ for every $t \geq 0$. Hence, the set of $\Phi$-minimal
solutions is not closed (or ``not stable'' or ``not upper
semicontinuous'') under pointwise convergence. Similarly, we may
consider $\delta>0$ to obtain a unique BV solution $q^\delta$ that
jumps at time $t=3$ before the unloading starts at $t=3{+}\delta>3$.
Clearly, these solutions are approximable and converge pointwise to
$q_1$, which is not approximable. Thus, the set of approximable
solutions is not upper semicontinuous.
\end{example}

\begin{example}\label{ex:5.4}\slshape
We provide an example where one BV solution corresponds to many
different parametrized metric solutions. The BV solution has
exactly one jump, and there are infinitely many distinct connecting
orbits $y$ in \ITEM{(iv)} of Definition~\ref{def:BVsln}, giving rise to
infinitely many distinct parametrized metric solutions. We consider
\[
\calQ=\R^2, \ \ \text{and} \ \  d(q,\wt{q})= \frac12 \left(|q_1
-\wt{q}_1|+|q_2 -\wt{q}_2| \right).
\]
With $q=(q_1,q_2)\in \calQ=\R^2$ the potential takes the form
\[
\calE(t,q) = U\left(\frac{q_1+q_2}{2}\right)+ W(q_1-q_2) -t
\left(\frac{q_1+q_2}{2}\right),
\]
where $U$ is defined in \eqref{def:u} and $W : \R \to [0,\infty)$ by
$W(\rho)=0$ for $|\rho|<1$ and $W(\rho)=(|\rho|-1)^2$ else. Starting
from $q(0)= (-5,-5)$, we have $q(t) = (\wt{q}(t),\wt{q}(t))$,
$\wt{q}$ being the BV solution of Example~\ref{ex:5.1}. Hence,
the (unique) jump occurs at $t=3$, starting in $(-2,-2)$ and ending
in $(6,6)$. However, the set of connecting paths $y$ is infinite.
Indeed, for every connecting path there holds  for a.a.\ $s \in
(s_0,s_1)$ \[|y'|(s) = \frac12 (|y_1'(s)| + |y_2'(s)|)\,, \quad
|\partial_q \calE(t,\cdot)|(q(s))= \frac12\left|U'\left(\frac{y_1(s)
+y_2 (s)}2\right) -t\right| \]if $|y_1(s) - y_2(s)| \leq 1$.
 Now, for a given curve $\gamma : [0,1]
\to [0,\infty)$ let us set $y_\gamma := (\wt{q}-\gamma,
\wt{q}+\gamma) $. Indeed, $|y_\gamma'|= 1/2 (|\wt{q}'-\gamma'| +
|\wt{q}'+\gamma'|)= |\wt{q}'|$ whenever $|\gamma'| \leq |\wt{q}'|$.
Therefore,
\[
\int_0^1 |\partial_q \calE(t,\cdot)|(y_\gamma(s)) |y_\gamma'|(s)\,
\rmd s =\int_0^1 |\partial_q \calE(t,\cdot)|(y(s)) |y'|(s)\dd s
\]
for all curves $\gamma$ with $\gamma(0)=\gamma(1) =0$ and
$|\gamma'|(s) \leq |\wt{q}'|(s)$ for a.a.\ $s \in (s_0,s_1)$, and for
such $\gamma's$ $y_\gamma$ is an optimal connecting curve.
\end{example}

\renewcommand{\baselinestretch}{0.95}
\small

\bibliographystyle{my_alpha}
\bibliography{alex_pub,bib_alex,MiRoSa_refs}

\newcommand{\etalchar}[1]{$^{#1}$}
\def\cprime{$'$} \def\ocirc#1{\ifmmode\setbox0=\hbox{$#1$}\dimen0=\ht0
  \advance\dimen0 by1pt\rlap{\hbox to\wd0{\hss\raise\dimen0
  \hbox{\hskip.2em$\scriptscriptstyle\circ$}\hss}}#1\else {\accent"17 #1}\fi}
\begin{thebibliography}{AA00}\itemsep0.1em

\bibitem[AGS05]{AmGiSa05GFMS}
{\scshape L.~Ambrosio, N.~Gigli, {\upshape and} G.~Savar{\'e}}.
\newblock {\em Gradient flows in metric spaces and in the space of probability
  measures}.
\newblock Lectures in Mathematics ETH Z\"urich. Birkh\"auser Verlag, Basel,
  2005.

\bibitem[AmD90]{AmbDMa90AGCR}
{\scshape L.~Ambrosio {\upshape and} G.~Dal~Maso}.
\newblock A general chain rule for distributional derivatives.
\newblock {\em Proc. Amer. Math. Soc.}, 108(3), 691--702, 1990.

\bibitem[BCS00]{BCS00IRFG}
{\scshape D.~Bao, S.-S.~Chern, {\upshape and} Z.~Shen}.
\newblock {\em An introduction to {R}iemann-{F}insler geometry}, volume 200 of
  {\em Graduate Texts in Mathematics}.
\newblock Springer-Verlag, New York, 2000.

\bibitem[BdV08]{BuDeVa08ECBG}
{\scshape M.~Buliga, G.~{de Saxc\'e}, {\upshape and} C.~Valle\'e}.
\newblock Existence and construction of bipotentials for graphs of multivalued
  laws.
\newblock {\em J. Convex Analysis}, 15(1), 87--104, 2008.

\bibitem[Cag08]{Cagn08?VVAF}
{\scshape F.~Cagnetti}.
\newblock A vanishing viscosity approach to fracture growth in a cohesive zone
  model with prescribed crack path.
\newblock {\em M3AS Math. Models Methods Appl. Sci.}, 2008.
\newblock To appear.

\bibitem[DD{\etalchar{*}}07]{DDMM07?VVAQ}
{\scshape G.~{Dal Maso}, A.~DeSimone, M.~Mora, {\upshape and} M.~Morini}.
\newblock A vanishing viscosity approach to quasistatic evolution in plasticity
  with softening.
\newblock {\em Arch. Rational Mech. Anal.}, 2007.
\newblock To appear.

\bibitem[DFT05]{DaFrTo05QCGN}
{\scshape G.~{Dal Maso}, G.~Francfort, {\upshape and} R.~Toader}.
\newblock Quasistatic crack growth in nonlinear elasticity.
\newblock {\em Arch. Rational Mech. Anal.}, 176, 165--225, 2005.

\bibitem[DGMT80]{DeMaTo80PEMS}
{\scshape E.~De~Giorgi, A.~Marino, {\upshape and} M.~Tosques}.
\newblock Problems of evolution in metric spaces and maximal decreasing curve.
\newblock {\em Atti Accad. Naz. Lincei Rend. Cl. Sci. Fis. Mat. Natur. (8)},
  68(3), 180--187, 1980.

\bibitem[DuS88]{DunSch58LOPI}
{\scshape N.~Dunford {\upshape and} J.~T.~Schwartz}.
\newblock {\em Linear operators. {P}art {I}}.
\newblock Wiley Classics Library. John Wiley \& Sons Inc., New York, 1988.

\bibitem[EfM06]{EfeMie06RILS}
{\scshape M.~Efendiev {\upshape and} A.~Mielke}.
\newblock On the rate--independent limit of systems with dry friction and small
  viscosity.
\newblock {\em J. Convex Analysis}, 13(1), 151--167, 2006.

\bibitem[Fed69]{Fede69GMT}
{\scshape H.~Federer}.
\newblock {\em Geometric measure theory}.
\newblock Die Grundlehren der mathematischen Wissenschaften, Band 153.
  Springer-Verlag New York Inc., New York, 1969.

\bibitem[FrM06]{FraMie06ERCR}
{\scshape G.~Francfort {\upshape and} A.~Mielke}.
\newblock Existence results for a class of rate-independent material models
  with nonconvex elastic energies.
\newblock {\em J. reine angew. Math.}, 595, 55--91, 2006.

\bibitem[Iof77]{Ioff77LSIF}
{\scshape A.~D.~Ioffe}.
\newblock On lower semicontinuity of integral functionals. {I}.
\newblock {\em SIAM J. Control Optimization}, 15(4), 521--538, 1977.

\bibitem[KMZ07]{KnMiZa07?ILMC}
{\scshape D.~Knees, A.~Mielke, {\upshape and} C.~Zanini}.
\newblock On the inviscid limit of a model for crack propagation.
\newblock {\em Math. Models Meth. Appl. Sci. (M$^3$AS)}, 2007.
\newblock Accepted (WIAS preprint 1268).

\bibitem[KrZ07]{KruZim07?EPNR}
{\scshape M.~Kru{\v{z}}{\'\i}k {\upshape and} J.~Zimmer}.
\newblock Evolutionary problems in non-reflexive spaces.
\newblock 2007.
\newblock Bath Institute for Complex Systems, Preprint~5/07.

\bibitem[MaM05]{MaiMie05EREM}
{\scshape A.~Mainik {\upshape and} A.~Mielke}.
\newblock Existence results for energetic models for rate--independent systems.
\newblock {\em Calc. Var. PDEs}, 22, 73--99, 2005.

\bibitem[MaM08]{MaiMie08?GERI}
{\scshape A.~Mainik {\upshape and} A.~Mielke}.
\newblock Global existence for rate-independent gradient plasticity at finite
  strain.
\newblock {\em J. Nonlinear Science}, 2008.
\newblock Submitted. WIAS preprint 1299.

\bibitem[Mie03]{Miel03EFME}
{\scshape A.~Mielke}.
\newblock Energetic formulation of multiplicative elasto--plasticity using
  dissipation distances.
\newblock {\em Cont. Mech. Thermodynamics}, 15, 351--382, 2003.

\bibitem[Mie05]{Miel05ERIS}
{\scshape A.~Mielke}.
\newblock Evolution in rate-independent systems ({C}h.~6).
\newblock In C.~Dafermos {\upshape and} E.~Feireisl, editors, {\em Handbook of
  Differential Equations, Evolutionary Equations, vol.~2}, pages 461--559.
  Elsevier B.V., Amsterdam, 2005.

\bibitem[MiT99]{MieThe99MMRI}
{\scshape A.~Mielke {\upshape and} F.~Theil}.
\newblock A mathematical model for rate-independent phase transformations with
  hysteresis.
\newblock In H.-D.~Alber, R.~Balean, {\upshape and} R.~Farwig, editors, {\em
  Proceedings of the Workshop on ``Models of Continuum Mechanics in Analysis
  and Engineering''}, pages 117--129, Aachen, 1999. Shaker-Verlag.

\bibitem[MiT04]{MieThe04RIHM}
{\scshape A.~Mielke {\upshape and} F.~Theil}.
\newblock On rate--independent hysteresis models.
\newblock {\em Nonl. Diff. Eqns. Appl. (NoDEA)}, 11, 151--189, 2004.
\newblock (Accepted July 2001).

\bibitem[MiZ08]{MieZel08?VVLP}
{\scshape A.~Mielke {\upshape and} S.~Zelik}.
\newblock On the vanishing viscosity limit in parabolic systems with
  rate-independent dissipation terms.
\newblock {\em In preparation}, 2008.

\bibitem[MRS08]{MiRoSa08?VVLM}
{\scshape A.~Mielke, R.~Rossi, {\upshape and} G.~Savar\'{e}}.
\newblock On the vanishing viscosity limit for the metric approach to
  rate-independent problems.
\newblock {\em In preparation}, 2008.

\bibitem[MTL02]{MiThLe02VFRI}
{\scshape A.~Mielke, F.~Theil, {\upshape and} V.~I.~Levitas}.
\newblock A variational formulation of rate--independent phase transformations
  using an extremum principle.
\newblock {\em Arch. Rational Mech. Anal.}, 162, 137--177, 2002.
\newblock (Essential Science Indicator: Emerging Research Front, August 2006).

\bibitem[NeO07]{NegOrt07?QSCP}
{\scshape M.~Negri {\upshape and} C.~Ortner}.
\newblock Quasi-static crack propagation by {G}riffith's criterion.
\newblock {\em Math. Models Methods Appl. Sci.}, 2007.
\newblock To appear.

\bibitem[RMS08]{RoMiSa08MACD}
{\scshape R.~Rossi, A.~Mielke, {\upshape and} G.~Savar\'{e}}.
\newblock A metric approach to a class of doubly nonlinear evolution equations
  and applications.
\newblock {\em Ann. Sc. Norm. Super. Pisa Cl. Sci. (5)}, VII, 97--169, 2008.

\bibitem[ToZ06]{ToaZan06?AVAQ}
{\scshape R.~Toader {\upshape and} C.~Zanini}.
\newblock An artificial viscosity approach to quasistatic crack growth.
\newblock {\em SISSA Preprint 43/M}, 2006.

\bibitem[Vis01]{Visi01NAE}
{\scshape A.~Visintin}.
\newblock A new approach to evolution.
\newblock {\em C.R.A.S. Paris}, 332, 233--238, 2001.

\bibitem[Vis06]{Visi06?MPE}
{\scshape A.~Visintin}.
\newblock A minimality principle for evolution.
\newblock Personal communication, 2006.

\end{thebibliography}

%\color{red}
%Comments on my changes:
%\begin{enumerate}
%\item I introduced the consitent usage of $\extQ:=[0,T]\ti \calQ$ for the
%  extended STATE space via \verb+ \extQ +.
%%%%%%%%%%%%%%%%%%%
%\item I substituted the old $S_0$ in Section 3 by $\Xi$, as now the distance
%  $S_\alpha$ is used in Section 4.
%\end{enumerate}

\end{document}